\documentclass{article} \usepackage{amsmath, amsthm, amsfonts,
  latexsym, amscd, amssymb, wasysym, graphicx, enumitem, stmaryrd}
\usepackage{bbm} % bbold (mathbb), dsfont (mathds), bbm (mathbbm)
\usepackage[mathscr]{eucal} \usepackage[letterpaper]{geometry}
\usepackage[colorlinks=true, linkcolor=black, citecolor=black,
urlcolor=blue]{hyperref} \usepackage[colorinlistoftodos]{todonotes}

\graphicspath{{figures/}}

\theoremstyle{theorem} \newtheorem{theorem}{Theorem}
\newtheorem{lemma}{Lemma} \newtheorem{proposition}{Proposition}

\theoremstyle{definition} \newtheorem{corollary}{Corollary}
\newtheorem{remark}{Remark} \newtheorem{definition}{Definition}

\newenvironment{proof_splt}{\textbf{Proof.}}{}

\newcommand{\Par}[1] {\left({#1}\right)} \newcommand{\Bra}[1]
{\left[{#1}\right]} \newcommand{\Cur}[1] {\left\{{#1}\right\}}
\newcommand{\Abs}[1] {\left|{#1}\right|}

\def \pd {\partial} \def \bd {\partial} \def \curl {\textup{curl}\,}
\def \div {\textup{div}\,} \def \Lap {\Delta} 
\def \meas {\textup{meas}\,} \def \dist {\textup{dist}\,}   \def \arg
{\mathrm{arg}\,} 

\def \Ind {\mathbbm{1}}

\def \Im {\textup{Im}\,} \def \Re {\textup{Re}\,}

\def \const {\text{const}}

 \def \Z {\mathbb{Z}}  \def \R
{\mathbb{R}} \def \C {\mathbb{C}}

\def \S {\mathbb{S}}

\def \eps {\varepsilon} \def \ve {\varepsilon} 

\def \index {\mathfrak{I}}

\def \d {{\delta}}  \def \he {{h_{ext}}} \def \heS
{{h_{ext}^2}}  \def \O {{\Omega}}  \def \Od {{\O_\d}} \def \odj {{\omega_\d^j}} \def \odi
{{\omega_\d^i}} \def \aj {{a^j}} % locations of the holes
\def \ai {{a^i}} % locations of the holes
\def \grj {{\gamma_{r}^j}} % curves around holes to accumulate energy
\def \gRj
{{\gamma_{R}^j}} % curves around holes where the degrees are defined
\def \odjo {{\overline{\omega}_\d^j}} \def \Rd {{R_\d}} \def \Rmi {{
    R_{\text{min}} }} \def \Rma {{ R_{\text{max}} }} \def \Brj
{{B_r^j}}  % union of rings around jth hole outside bad disks
\def \Tj {{B(\aj, R/2) \setminus \overline{B(\aj,
      R/4)}}} % annulus around jth hole where $\phi_j$ is non-constant

\def \bi {{b^i}} % locations
\def \bdi {{d_i}} % degrees
\def \bdsum {{d}} % total degree

\def \dj {{D^j}} \def \dd {{D_\d}} \def \ddj {{D_\d^j}} \def \dedj
{{D_{\d,\eps}^j}}  
 \def \dvi {{D_v^i}} \def \dvj {{D_v^j}} \def \dv
{{D_v}} \def \drj {{D_r^j}} \def \drjp {{\Par{D_r^j}}}

\def \djp {{\Par{D^j}}} \def \dvjp {{\Par{D_v^j}}}

\def \udd {{u_{\d D}}} \def \ud {{u_\d}} \def \ude {{u_\d^{\eps}}}
\def \Add {{A_{\d D}}} \def \Ad {{A_\d}} \def \Ade {{A_\d^{\eps}}}
\def \hdd {{h_{\d D}}} \def \hd {{h_\d}} \def \hde {{h_\d^{\eps}}}

\def \zisb {{\zeta_i^{\text{sub}}}} \def \zisp
{{\zeta_i^{\text{sup}}}} \def \zijsp {{\zeta_{ij}^{\text{sup}}}}

\def \Hrj {{H_{R}^j}} \def \Hj {{H_j}} \def \Htj {{\widetilde{H}_j}}
\def \cj {{c_j}} \def \ci {{c_i}} \def \fj {{f_j}} \def \Adj
{{A_\d^j}}

\def \uddo {{\overline{u}_{\d D}}} \def \Pde {{\Pi_\d^{\eps}}}

\title{On approximation of Ginzburg-Landau minimizers by $\S^1$-valued
  maps in domains with vanishingly small holes}

\author{Leonid Berlyand\thanks{Department of Mathematics, Pennsylvania
    State University, \tt{berlyand@math.psu.edu},
    \tt{iaroshenko@psu.edu}}, Dmitry Golovaty\thanks{Department of
    Mathematics, The University of Akron, \tt{dmitry@uakron.edu}},
  Oleksandr Iaroshenko\footnotemark[1], Volodymyr
  Rybalko\thanks{Mathematical Division, B. I. Verkin Institute for Low
    Temperature, Physics and Engineering, Kharkiv, Ukraine,
    \tt{vrybalko@ilt.kharkov.ua}}} \date{\today}

\begin{document}
\maketitle
%	\tableofcontents
%	\listoftodos \vspace{4mm}
	
Abstract: We consider a two-dimensional Ginzburg-Landau problem on an
arbitrary domain with a finite number of vanishingly small circular
holes. A special choice of scaling relation between the material and
geometric parameters (Ginzburg-Landau parameter vs hole radius) is
motivated by a recently discovered phenomenon of vortex phase
separation in superconducting composites. We show that, for each hole,
the degrees of minimizers of the Ginzburg-Landau problems in the
classes of $\mathbb S^1$-valued and $\mathbb C$-valued maps, respectively, are the
same. The presence of two parameters that are widely separated on a
logarithmic scale constitutes the principal difficulty of the analysis
that is based on energy decomposition techniques.

\section{Introduction}

The present study is motivated by the pinning phenomenon in type-II
superconducting composites. Type-II superconductors are characterized
by vanishing resistivity and complete expulsion of magnetic fields
from the bulk of the material at sufficiently low temperatures. When
the magnitude $h_{ext}$ of an external magnetic field
$\mathbf{h}_{ext}$ exceeds a certain threshold, the field begins to
penetrate the superconductor along isolated vortex lines that may
move, resulting in energy dissipation. This motion and related energy
losses can be inhibited by pinning the lines to impurities or holes in
a \textit{superconducting composite}. Understanding the role of
imperfections in a superconductor can thus be used to design more
efficient superconducting materials. In what follows, we will consider
a cylindrical superconducting sample containing rod-like inclusions or
\textit{columnar defects} elongated along the axis of the cylinder, so
that the sample can be represented by its cross-section
$\O\subset\mathbb{R}^2$. Then the vortex lines penetrate each
cross-section at isolated points, called \textit{vortices}.
	
Superconductivity is typically modeled within the framework of the
Ginzburg-Landau theory \cite{GinLan65} in terms of an order-parameter
$u\in\mathbb C$ and the vector potential of the induced magnetic field
$A\in\mathbb R^2$. The appearance and behavior of vortices for the
minimizers of the Ginzburg-Landau functional
\begin{equation}\label{eq:full_GL_func_1}
  GL^\ve[u, A]
  = \frac12\int_{\Omega}|\left(\nabla - iA\right)u|^2\,dx
  + \frac{1}{4\ve^2}\int_{\Omega}(1-|u|^2)^2\,dx
  + \frac12\int_{\Omega}(\curl A - \he)^2\,dx
\end{equation}
have been studied, in particular, in \cite{SanSer00, SanSer03} where
the existence of two critical magnetic fields, $H_{c1}$ and $H_{c2}$,
was established rigorously for simply-connected domain when $\ve>0$ is small. When the external
magnetic field is weak ($\he < H_{c1}$) it is completely expelled from
the bulk semiconductor (Meissner effect) and there are no
vortices. When the field strength is ramped up from $H_{c1}$ to
$H_{c2}$, the magnetic field penetrates the superconductor through an
increasing number of isolated vortices while the superconductivity is
destroyed everywhere, once the field exceeds $H_{c2}$.
		
The pinning phenomenon that we consider in this paper is observed in
non-simply-connected domains with holes that may or may not contain
another material. If a hole "pins" a vortex, the order parameter $u$
has a nonzero winding number on the boundary of the hole. We refer to
this object as a \textit{hole vortex}. Note that degrees of the hole
vortices increase along with the strength of the external magnetic
field. This situation is in contrast with the regular bulk vortices
that have degree $\pm1$ and increase in number as the field becomes
stronger.

An alternative way to model the impurities is to consider a potential
term $(a(x)-|u|^2)^2$ where $a(x)$ varies throughout the sample. It
was proven in \cite{ChaRic97} that the impurities corresponding to the
weakest superconductivity (where $a(x)$ is minimal) pin the vortices
first. This model was studied further in \cite{AftSanSer01} and
\cite{AndBauPhi03} to demonstrate the existence of nontrivial pinning
patterns and in \cite{AlaBro05} to investigate the breakdown of
pinning in an increasing external magnetic field, among other
issues. A composite consisting of two superconducting samples with
different critical temperatures was considered in \cite{AydKac09,
  Kac10} where nucleation of vortices near the interface was shown to
occur.

In our model we consider a superconductor with holes, similar to the
setup in \cite{AlaBro06}. In that work, the authors considered the
asymptotic limits of minimizers of $GL^\eps$ as $\eps\to0$ and
determined that holes act as pinning sites gaining nonzero degree for
moderate but bounded magnetic fields. For magnetic fields below the
threshold of order $|\ln\eps|$ the degree of the order parameter on
the holes continues to grow without bound, however beyond the critical
field strength, the pinning breaks down and vortices appear in the
interior of the superconductor. Since the contribution to the energy
from the hole vortices has a logarithmic dependence on the diameter of
the holes, the hole size can be used as an additional small parameter
to enforce a finite degree of the hole vortex in the limit of small
$\eps$. The domain with finitely many shrinking (pinning) subdomains
with weakened superconductivity was considered in \cite{DosMis11} in
the case of the simplified Ginzburg-Landau functional.  The model with
a potential term $(a(x)-|u|^2)^2$ with piecewise constant $a(x)$ was
used to enforce pinning and it was observed that the vortices are
localized within pinning domains and converge to their centers.

The problem considered in this work was inspired by the result in
\cite{BerRyb13} where a periodic lattice of vanishingly small holes
was considered. The main interest was in the regime when the radii of
the holes were exponentially small compared to the period $a$ of the
lattice; both of these parameters were assumed to converge to zero
along with $\eps$. Using homogenization-type arguments, it was shown
in \cite{BerRyb13} that in the limit of $\eps\to 0$ and when the
external magnetic field of order $O(a^{-2})$, the minimizers can be
characterized by nested subdomains of constant vorticity. The physical
nature of this result was discussed in \cite{IarBerRybVin13}. The
analysis in \cite{BerRyb13} relies on a conjecture that for small
$\eps$, the degrees of the hole vortices are the same for both $\C$-
and $\S^1$-valued maps. The principal aim of the present paper is to
establish the validity of this conjecture in the case of finitely many
vanishingly small holes.

Our approach builds on that of \cite{AlaBro06}, combined with the
appropriately chosen lower bounds on the energy and the ball
construction method \cite{BetBreHel94}, \cite{Jer99},
\cite{San98}-\nocite{SanSer00, SanSer00-2, SanSer03}\cite{SanSer07}.

The paper is organized as follows. Section \ref{sec:main_results}
contains the formulation of the problem, as well as the main result
described in Theorem \ref{thm:main}. In Section
\ref{sec:s1_valued_problem}, we prove that the minimizers in the class
of $\S^1$-valued maps are characterized by the unique set of integer
degrees on the holes. In Section \ref{sec:energy_decoposition}, we use
the approach, similar to that in \cite{AlaBro06}, to express the
energy of a $\C$-valued minimizer as the sum of the the energy of the
$S^1$-valued minimizer and the remainder terms. Compared to
\cite{AlaBro06}, additional complications arise in the analysis due to
the fact that the radius of the holes is not fixed in the present
work. In particular, because of the presence of another small
parameter, we use a different ball construction method that
incorporates both the Ginzburg-Landau parameter $\eps$ and the holes
radius $\d$. In Section \ref{sec:absence_of_bulk_vortices} we show
that the minimizes cannot have vortices with nonzero degrees outside
of the holes. This section also provides sharp energy estimates that
allow us to prove the main theorem. Finally, in Section
\ref{sec:equality_of_degrees}, the equality of degrees is established based on the estimates obtained
in the previous section.

\section{Main results}\label{sec:main_results}

Let $B\left(x_0, R\right)\subset\mathbb{R}^2$ denote a disk of radius
$R$ centered at $x_0$. Let $\O$ be an arbitrary smooth, bounded,
simply connected domain and suppose that
$\odj = B(\aj, \d) \subset \O$, $j = 1\dots N$ represent the holes in
$\Omega$, where $\aj$ is the center of the hole $j=1,\ldots,N$ and
$\d\ll1$ is its radius. We introduce the perforated domain
\begin{equation}
  \Od = \O \setminus \bigcup_{j=1}^N \odj
\end{equation}
and consider the Ginzburg-Landau functional
\begin{equation}\label{eq:full_GL_func}
  GL^\eps_\d[u, A]
  = \frac12\int\limits_{\Od}|\left(\nabla - iA\right)u|^2\,dx
  + \frac{1}{4\eps^2}\int\limits_{\Od}(1-|u|^2)^2\,dx
  + \frac12\int\limits_{\O}(\curl A - \he)^2\,dx.
\end{equation} 
The domain $\Od$ represents a cross-section of a superconducting sample. Here $u:\Od\to\mathbb{C}$ is an order parameter, $A:\O\to\mathbb{R}^2$ is a vector potential of the induced magnetic field, and $\he$ is the magnitude of the external magnetic field. By $\eps$ we denote the inverse of the Ginzburg-Landau parameter that determines the radius of a typical vortex core. In what follows, we will assume that the cores radii are much smaller than the radius of the holes $\odj$.
    
The functional $GL^\eps_\d[u, A]$ is gauge-invariant, i.e., for any
$\varphi\in H^2(\Omega, \mathbb{R})$ and any admissible pair $(u, A)$,
the equality
$GL^\eps_\d[u, A] = GL^\eps_\d\left[u\, e^{i\varphi}, A +
  \nabla\varphi\right]$ always holds. This degeneracy can be
eliminated by imposing the \textit{Coulomb gauge}, that is requiring
that
\begin{equation}\label{eq:gauge}
  A \in H(\Omega,\mathbb{R}^2)
  := \left\{ a\in H^1(\Omega, \mathbb{R}^2) \mid \div{a}=0 \text{ in } \Omega,
    \ a\cdot\nu=0 \text{ on } \partial\Omega \right\},
\end{equation}
where $\nu$ is an outward unit normal vector to $\partial\Omega$. We
will fix the Coulomb gauge throughout the rest of this work.
    
We consider the minimizers of the two variational problems
\begin{equation}\label{eq:min_prob}
  \Par{\ude, \Ade}
  := \arg\min\Cur{GL^\eps_\d[u, A] \mid u \in H^1(\Od; \C), A \in H(\O; \R^2)},
\end{equation}
and
\begin{equation}\label{eq:min_prob_s1.0}
  \Par{\ud, \Ad}
  := \arg\min\Cur{GL_\d^\eps[u, A] \mid u \in H^1(\Od; S^1), A \in H(\O; \R^2)}.
\end{equation}
Note that, trivially,
\begin{equation}\label{eq:min_prob_s1}
  \Par{\ud, \Ad}
  := \arg\min\Cur{GL_\d[u, A] \mid u \in H^1(\Od; S^1), A \in H(\O; \R^2)},
\end{equation}
where
\begin{equation}\label{eq:S1_func}
  GL_\d[u, A]
  = \frac12\int\limits_{\Od} |\nabla u - iAu|^2\,dx
  + \frac12\int\limits_\O(\curl A - \he)^2\,dx.
\end{equation}
    
For any hole center $a^j,\ j=1,\ldots,N$ and $R>0$, let
$\gRj = \bd B(\aj, R)$ be a circle of radius $R$ centered at $a^j$. In
what follows we make a frequent use of the following
    
\begin{definition}
  Given a $u\in H^1\left(\Omega_\delta,\mathbb C\right)$ and
  $a^j,\ j=1,\ldots,N$, suppose there exists an $R = \d + o(\d)$ such
  that the winding number $d = \deg{(u/|u|, \gRj)}\neq0$.  Then $u$ is
  said to have a {\it hole vortex} of the degree $d$ inside $\odj$.
\end{definition}

The existence of $\gRj$ is established in the Theorem \ref{thm:main}
and they are specified using the results of Theorem
\ref{thm:bcm}. Hole vortices may exist inside $\odj$ for the
minimizers of both \eqref{eq:min_prob} and \eqref{eq:min_prob_s1} and
our principal goal is to prove that the respective degrees of the hole
vortices arising in both problems coincide for the same external
magnetic field as long as the parameter $\d$ is sufficiently
small. This result implies that the non-linear potential term can be
effectively replaced by the constraint $|u| = 1$ when one is
interested in studying the distribution of degrees of the hole
vortices for the minimizer of the problem \eqref{eq:min_prob}.
    
The main result of this work is the following theorem.
\begin{theorem}\label{thm:main}
  Assume that the parameters $\eps$ and $\delta$ satisfy
  \begin{equation}\label{eq:cond_eps}
    |\log\eps| \gg |\log\d|.
  \end{equation}
  Suppose
  \begin{equation}\label{eq:uniqueness_cond}
    \sigma \in \R_+\setminus\Sigma
  \end{equation}
  where $\Sigma$ is a discrete set described below. Let
  \begin{equation}\label{eq:cond_he}
    \he = \sigma|\log\d|
  \end{equation}
  and $\Par{\ude, \Ade}$ and $\Par{\ud, \Ad}$ be defined by
  \eqref{eq:min_prob} and \eqref{eq:min_prob_s1}, respectively.
        
  Then, for a sufficiently small $\d$, there exists an
  $\Rd\in[\d,\d+\d^2]$ such that 
  \begin{enumerate}
\item[\bf (i)] $\ddj = \deg{\left(\ud, \gRj \right)}$ coincide for all
  $j = 1\dots N$ when $\dedj$ are defined, e.g. when $\ude \neq 0$ on
  $\gRj$;
  \item[\bf (ii)] the degrees of the hole vortices
  $\dedj = \deg{\left(\frac{\ude}{\left|\ude\right|}, \gRj \right)}$;
\end{enumerate}
for any $R \geq \Rd$ for which
  $\gRj = \bd B(\aj, R),\ j = 1\ldots N$ are mutually disjoint and do
  not intersect $\bd\O$.
\end{theorem}
\begin{remark}
  The set $\Sigma$ includes the appropriately scaled values of the
  external field at which the degree of one of the hole vortices
  increments by one, i.e. from $d$ to $d+1$. At these threshold field
  strengths, the leading order approximation of the energy is the same
  for both degrees $d$ and $d+1$ and the degrees of the hole vortices
  of minimizers $\ude$ and $\ud$ cannot be determined uniquely. The
  set $\Sigma$ is described as follows:
  \begin{equation}
    \Sigma = \bigcup_{j=1}^N \Sigma_j
    \quad\text{where}\quad
    \Sigma_j = \Cur{\sigma > 0 \mid \sigma\Par{1 - \xi_0(\aj)} \in \Z + \frac12}
  \end{equation}
  consists of the threshold field values for the hole $j = 1\dots N$ and
  the function $\xi_0$ solves the boundary value problem
  \begin{equation}\label{eq:xi_0}
    \begin{cases}
      -\Lap \xi_0 + \xi_0 = 0 &\text{ in } \O, \\
      \xi_0 = 1 &\text{ on } \bd\O.
    \end{cases}
  \end{equation}
\end{remark}
\begin{remark}
  Notice that, since $\ud(x) \in \S^1$, there are no vortices outside
  of the holes and thus
  \begin{equation}
    \ddj = \deg{\Par{\ud, \grj}} = \deg{\Par{\ud, \bd\odj}}
  \end{equation}
  for all $j = 1\ldots N$.
\end{remark}
\begin{remark}
  As we will show in Section \ref{sec:absence_of_bulk_vortices},
  although the external magnetic field satisfying the bound
  \eqref{eq:cond_he} is strong enough to generate hole vortices, it is
  too weak for vortices to appear inside the bulk superconductor
  $\Od$, away from the boundary $\bd\O$.
\end{remark}
    
We prove Theorem \ref{thm:main} in two steps. First, we consider
minimizers $\Par{\udd, \Add}$ of the variational problem
\eqref{eq:S1_func} in the class of $\S^1$-valued maps with the
prescribed degrees, $\deg(u, \bd\odj) = \dj$, $j = 1\dots N$, by
setting
\begin{equation}\label{eq:min_prob_s1_prescribed_D}
  \Par{\udd, \Add}
  := \arg\min\Cur{GL_\d[u, A] \mid u \in H^1(\Od; S^1), A \in H(\O; \R^2), \deg(u, \bd\odj) = \dj}.
\end{equation} Then the degrees $\ddj$ of the map $\ud$ minimize the energy
\begin{equation}\label{eq:energy_as_l_of_D}
  l_\d(D) := GL_\d\left[\udd, \Add\right]
\end{equation}
where $D = (D^1, \dots, D^N)$. It turns out that the function
$l_\d(D)$ is a quadratic polynomial in $D^1,\dots,D^N$. Its minimum is
attained at one of the integer points adjacent to the vertex of
paraboloid $l_\d(T)$ with $T \in \R^N$. We enforce the condition
\eqref{eq:uniqueness_cond} to ensure that such minimizing integer
point is unique.
        
We then express a minimizer $\Par{\ude, \Ade}$ of $GL_\d^\eps[u, A]$
as a sum of $\Par{\ud, \Ad}$ and an appropriate correction term and
consider a corresponding energy decomposition in the spirit of the
approach in \cite{AlaBro06} for finite-size holes. The analysis relies
principally on the techniques developed in \cite{AlaBro06} and the
ball construction method \cite{SanSer07}. Compared to \cite{AlaBro06},
new challenges arise due to the presence of the second small parameter
that require additional estimates and sharper energy bounds.

\section{$S^1$-valued problem}\label{sec:s1_valued_problem}

The main goal of this section is to establish the relation between the
energy of the minimizer $\Par{\udd, \Add}$ and the degrees $D$ of the
hole vortices corresponding to $\udd$. We approximate the minimizer
$\Par{\udd, \Add}$, calculate its energy
$l_\d(D) = GL_\d\Bra{\udd, \Add}$, and find the minimizing degrees
$\dd = (D_\d^1, \dots, D_\d^N)$. We prove the following theorem.

\begin{theorem}\label{thm:S1}
  Let $\Par{\udd, \Add}$ be a minimizer of
  \eqref{eq:min_prob_s1_prescribed_D} with the prescribed degrees
  $D \in \Z^N$ on the holes. Then the Ginzburg-Landau energy
  $GL_\d\left[\udd, \Add\right]$, expressed as a function of $D$,
  takes the following form:
  \begin{equation}\label{eq:l_expansion}
    l_\d(D)
    = \pi\sum_{j=1}^N \Bra{ \djp^2 - 2\sigma\Par{1 - \xi_0(\aj)}\dj }|\log\d|
    + C|\log\d|^2
    + |D|^2 O(1)
  \end{equation}
  where $\xi_0$ solves the boundary value problem \eqref{eq:xi_0},
  $C = O(1)$, and $|D| = \max\limits_j|\dj|$.
\end{theorem}
\begin{proof_splt}
  The main idea of the proof is to approximate the induced magnetic field $\hdd=\curl{\Add}$ as a
  sum of functions that depend on external magnetic field and the
  prescribed degrees on the holes, respectively. First, prescribe the
  degrees of the order parameter
  \begin{equation}
  \label{eq:18}
    \deg(u, \bd\odj) = \dj, \quad j = 1\dots N
  \end{equation}
  and write down the Euler-Lagrange equation for \eqref{eq:S1_func} in
  terms of the induced magnetic field $h = \curl{A}$ with the
  corresponding boundary conditions:
  \begin{equation}\label{eq:magn_field}
    \begin{cases}
      -\Lap h + h = 0, &\text{ in } \Od, \\
      h = \he, &\text{ on } \bd\O, \\
      h = \Hj, &\text{ in } \odj, \quad j = 1\dots N, \\
      -\int_{\bd\odj} \frac{\pd h}{\partial \nu} ds = 2\pi\dj -
      \int_{\odj} h dx, &\ j = 1\dots N.
    \end{cases}
  \end{equation}
  The constants $H_j$ are a priori unknown and are defined through the
  solution $\hdd = \hd(D)$ of \eqref{eq:magn_field} where
  $D = (D^1, \dots, D^N)$ is the vector of the prescribed degrees. The
  energy \eqref{eq:S1_func} of the minimizer $\Par{\udd, \Add}$ can be
  expressed in terms of $\hdd$:
  \begin{equation}\label{eq:S1_func_h}
    GL_\d\left[\udd, \Add\right]
    = GL_\d[\hdd]
    = \frac12\int_{\Od}|\nabla\hdd|^2dx + \frac12\int_{\O}(\hdd - \he)^2dx\,.
  \end{equation}
  Decompose the solution of \eqref{eq:magn_field} $\hdd$ into
  \begin{equation}\label{eq:magn_field_decomposition}
    \hdd = h_1 + h_2 + h_3,
  \end{equation}
  where $h_1$ captures the influence of the external field $\he$,
  $h_2$ takes into account the hole vortices, and $h_3$ is the
  remainder. More precisely,
  \begin{equation}
    h_1 = \he\xi_0,
  \end{equation}
  where $\xi_0$ solves the boundary value problem \eqref{eq:xi_0} in
  the domain $\O$ with no holes:
  \begin{equation}\label{eq:xi_0_1}
    \begin{cases}
      -\Lap \xi_0 + \xi_0 = 0 &\text{ in } \O, \\
      \xi_0 = 1 &\text{ on } \bd\O.
    \end{cases}
  \end{equation}
  The function $h_2$ is defined by
  \begin{equation}\label{eq:h_2}
    h_2(x) 
    = \sum_{j=1}^N \dj \theta_j(x)\phi_j(x)
  \end{equation}
  where $\dj$ are as in \eqref{eq:18}. Here
  $$\theta_j(x) = \theta(x - \aj),\ j=1,\ldots,N$$ and $\theta$ is a truncated modified Bessel function of
  the second kind
  \begin{equation}
    \theta(x) =
    \begin{cases}
      K_0(\d), &|x| \leq \d, \\
      K_0(|x|), &|x| > \d.
    \end{cases}
  \end{equation}
  The cutoff function $\phi_j(x) = \phi(x - \aj) \in C^{\infty}(\R^2)$
  satisfies
  \begin{equation}
    \phi(x) =
    \begin{cases}
      1, &|x| \leq R/4, \\
      0, &|x| \geq R/2,
    \end{cases}
  \end{equation}
  with $R$ being defined as the largest radius for which $B(\aj, R)$,
  $j = 1\dots N$ intersect neither each other nor the boundary
  $\bd\O$. Here the choice of $K_0(|x|)$ is motivated by the fact that
  it is a fundamental solution of the equation
  $-\Lap u + u = 2\pi\delta(x)$ in $\R^2$. Note that $h_2$ solves the
  following problem:
  \begin{equation}\label{eq:h_2^i_bvp}
    \begin{cases}
      -\Lap h_2 + h_2 = \sum_{j=1}^{N}\dj\Bra{-\Lap + I}\Par{\theta_j\phi_j}, &\text{ in } \Od, \\
      h_2 = 0, &\text{ on } \bd\O, \\
      h_2 = \dj K_0(\d), &\text{ on } \bd\odj, \quad j = 1\dots N, \\
      -\int_{\bd\odj} \frac{\pd h_2}{\partial \nu} ds = 2\pi\dj - \dj
      K_0(\d)|\odj| + \dj O(\delta^2), &\ j = 1\dots N.
    \end{cases}
  \end{equation}
  Since for each $j=1,\ldots,N$ the function
  $f_j(x) := \Bra{-\Lap + I}\Par{\theta_j\phi_j}$ is nonzero only
  inside the annular region $T_j: = \Tj$ that does not intersect any
  of the holes, the functions $f_j,\ j=1,\ldots,N$ are smooth and
  finite. Thus, for every $j=1,\ldots,N$, the function $h_2$ has the
  degree $\dj$ on the hole $\odj$ and $\theta_j\phi_j$ is constant on
  $\odj$ and decays to zero on $\bd B(\aj, R/2)$.

  Next, we show that the contribution of the remainder
  $h_3 = h - h_1 - h_2$ to the energy is small, hence the interaction
  between the hole vortices contributes a negligible amount to the
  energy. This provides a justification for treating each hole vortex
  as being independent from the other hole vortices.
		
  We deduce the boundary value problem for $h_3$ from the original
  problem \eqref{eq:magn_field}, the problem \eqref{eq:xi_0} for
  $h_1 = \he\xi_0$, and the expression \eqref{eq:h_2} for $h_2$ to
  obtain:
  \begin{equation}\label{eq:h_3}
    \begin{cases}
      -\Lap h_3 + h_3 = -\sum_{j=1}^N \dj\fj(x), &\text{ in } \Od, \\
      h_3 = 0, &\text{ on } \bd\O, \\
      h_3 = \Htj - \he(\xi_0(x) - \xi_0(\aj)), &\text{ on } \bd\odj, \quad j = 1\dots N, \\
      -\int_{\bd\odj} \frac{\pd h_3}{\partial \nu} ds = -\Htj|\odj| +
      \dj O(\d^2) + O(\d^3\log\d), &\ j = 1\dots N.
    \end{cases}
  \end{equation}
  where $\Htj = \Hj - \he\xi_0(\aj) - \dj K_0(\d)$ are the unknown
  constants. The next lemma establishes the necessary estimates for
  $h_3$.
\end{proof_splt}

\begin{lemma}
  The solution $h_3$ of \eqref{eq:h_3} satisfies the following
  estimates:
  \begin{align}
    \|h_3\|_{L^\infty(\O)}
    &\leq C_1\d|\log\d|^2 + C_2|D|, \\
    \|\nabla h_3\|_{L^\infty(\O)}
    &\leq C_1|\log\d|^2 + C_2|D||\log\d|, \\
    \Abs{\frac{\pd h_3}{\pd \nu}}
    &\leq C_1|\log\d| + C_2|D|
      \text{ on }
      \bd\odj
      \text{ for all }
      j = 1\dots N.
  \end{align}
\end{lemma}
\begin{proof}
  We begin by splitting \eqref{eq:h_3} into several
  subproblems. First, let $\eta = \sum_{j=1}^N \dj\eta_j$ be a
  solution of the nonhomogeneous equation in \eqref{eq:h_3}, where
  $\eta_j$ solves
  \begin{equation}\label{eq:eta_j}
    \begin{cases}
      -\Lap \eta_j + \eta_j = -\Bra{-\Lap + I}\Par{\theta_j\phi_j}\Ind_{T_j}, &\text{ in } \O, \\
      \eta_j = 0, &\text{ on } \bd\O,
    \end{cases}
  \end{equation}
  for every $j=1,\ldots,N$. Here $\eta_j,\ j=1,\ldots,N$ are smooth and
  do not depend on $\d$. Next, introduce $\eta_0$ that both solves
  the homogeneous equation and satisfies the conditions on $\bd\odj$
  in \eqref{eq:h_3} to give
  \begin{equation}\label{eq:eta_0}
    \begin{cases}
      -\Lap \eta_0 + \eta_0 = 0, &\text{ in } \Od, \\
      \eta_0 = 0, &\text{ on } \bd\O, \\
      \eta_0 = -\he(\xi_0(x) - \xi_0(\aj)) - (\eta(x) - \eta(\aj)),
      &\text{ on } \bd\odj, \quad j = 1\dots N.
    \end{cases}
  \end{equation}
  Note that, by the Maximum Principle,
  \begin{equation}\label{eq:est_eta_0}
    \|\eta_0\|_{L^\infty} \leq C\d(|\log\d| + \max_j|\dj|).
  \end{equation}
  Lemma \ref{lem:grad_est} provides the estimate on the gradient of
  $\eta_0$ of the form
  \begin{equation}\label{eq:est_grad_eta_0}
    \|\nabla\eta_0\|_{L^\infty} \leq C(|\log\d| + \max_j|\dj|).
  \end{equation}
  The remainder $\zeta = h_3 - \sum_{j=0}^{N}\eta_j$ solves the
  following system:
  \begin{equation}\label{eq:zeta}
    \begin{cases}
      -\Lap \zeta + \zeta = 0, &\text{ in } \Od, \\
      \zeta = 0, &\text{ on } \bd\O, \\
      \zeta = \cj, &\text{ on } \bd\odj, \quad j = 1\dots N, \\
      -\int_{\bd\odj} \frac{\pd \zeta}{\partial \nu} ds = -|\odj|\cj +
      \Adj, &\ j = 1\dots N,
    \end{cases}
  \end{equation}
  where $\cj = \Htj - \eta(\aj)$ are unknown constants and
  $\Adj = |D| O(\d) + O(\d\log\d)$ is an error. The first three
  equations in \eqref{eq:zeta} set up the boundary value problem for
  $\zeta$ with the unknown boundary values $\cj$. The fourth line in
  \eqref{eq:zeta} gives the system of $N$ equations for $N$ unknowns
  $\cj$. Since the boundary value problem for $\zeta$ is linear, we
  start with the estimates for the basis functions $\zeta_i$ that
  solve the problem
  \begin{equation}\label{eq:zeta_i}
    \begin{cases}
      -\Lap \zeta_i + \zeta_i = 0, &\text{ in } \Od, \\
      \zeta_i = 0, &\text{ on } \bd\O, \\
      \zeta_i = \delta_{ij}, &\text{ on } \bd\odj, \quad j = 1\dots N,
    \end{cases}
  \end{equation}
  for every $i=1,\ldots,N$. Then, using representation
  $\zeta = \sum_i \ci\zeta_i$, we will solve the linear system for
  $\ci$.
		
  We use the method of sub- and supersolutions to get estimates for
  $\zeta_i$. By the Maximum Principle, we have that
  $0 \leq \zeta_i \leq 1$ for every $i=1,\ldots,N$. In the case of a
  radially symmetric domain with one hole at the center, the solutions
  of \eqref{eq:zeta_i} are the modified Bessel functions. We show that
  they provide a good approximation for $\zeta_i$. First, fix
  $i \in 1\dots N$ and construct a supersolution for $\zeta_i$. Take
  $\Rma > 0$ such that $\O \in B(\ai, \Rma)$ and set
  \begin{equation}
    \zisp
    = \frac{K_0\Par{\frac{|x - \ai|}{\Rma}}}{K_0\Par{{\frac{\d}{\Rma}}}}.
  \end{equation}
  The function $\zisp$ is strictly positive in $\Od$, equals $1$ on
  $\bd\odi$, and has $\Bra{-\Lap + I}\zisp = 0$. Therefore it
  satisfies
  \begin{equation}
    \begin{cases}
      -\Lap\zisp + \zisp = 0 &\text{ in } \Od, \\
      \zisp > 0 &\text{ on } \bd\O, \\
      \zisp = 1 &\text{ in } \odi, \\
      \zisp > 0 &\text{ in } \odj,\ j \neq i, \quad j = 1\dots N,
    \end{cases}
  \end{equation}
  and is thus a supersolution. This yields the bound
  \begin{equation}\label{eq:sup_sln_ineq}
    0 \leq \zeta_i \leq \zisp\text{ in } \O, \quad i = 1\dots N.
  \end{equation}
  Next, we construct a subsolution. Take $\Rmi > 0$ such that
  $B(\ai, 2\Rmi) \in \Od$ for every $i = 1\dots N$ and set
  \begin{equation}
    \zisb
    = \frac{K_0\Par{\frac{|x - \ai|}{\Rmi}}}{K_0\Par{\frac{\d}{\Rmi}}}
  \end{equation}
  The Bessel function is a fundamental solution of
  $\Bra{-\Lap + I}u = \delta(x)$ and it is decreasing, therefore
  $\zisb$ is negative outside $B(\ai, \Rmi)$. Thus it satisfies
  \begin{equation}
    \begin{cases}
      -\Lap\zisb + \zisb = 0 &\text{ in } \Od, \\
      \zisb < 0 &\text{ on } \bd\O, \\
      \zisb = 1 &\text{ in } \odi, \\
      \zisb < 0 &\text{ in } \odj,\ j \neq i, \quad j = 1\dots N,
    \end{cases}
  \end{equation}
  and is thus a subsolution. This, together with
  \eqref{eq:sup_sln_ineq}, implies that
  \begin{equation}
    \max(0, \zisb) \leq \zeta_i \leq \zisp,
  \end{equation}
  for every $i = 1\dots N$, giving a very sharp description of the
  behavior of $\zeta_i$ near $i$th hole. Note that, for
  $x \in \bd\odi$, we have
  \begin{equation}
    \frac{L_1}{\d\log\d}
    \leq \frac{\pd\zisb}{\pd\nu}(x)
    \leq \frac{\pd\zisp}{\pd\nu}(x)
    \leq \frac{L_2}{\d\log\d}
  \end{equation}
  with $L_1, L_2 > 0$, therefore
  \begin{equation}\label{eq:precise_behavior}
    \frac{\pd\zeta_i}{\pd\nu}(x)
    \sim \frac{1}{\d\log\d}
    \text{ on }
    \bd\odi.
  \end{equation}
  To estimate the normal derivative of $\zeta_i$ on $\bd\odj$ for
  $j \neq i$ we need a better supersolution that captures the
  appropriate Dirichlet boundary conditions. Outside of
  $B(\ai, \Rmi)$, we have
  \begin{equation}
    |\zeta_i(x)|
    \leq \frac{K_0\Par{\frac{\Rmi}{\Rma}}}{K_0\Par{{\frac{\d}{\Rma}}}}
    \leq C_R|\log\d|^{-1}.
  \end{equation}
  Construct $\zijsp$ that solves the following conditions:
  \begin{equation}\label{eq:zijsp}
    \begin{cases}
      -\Lap\zijsp + \zijsp = 0 &\text{ in } B(\aj, \Rmi) \setminus \overline{B(\aj, \d)}, \\
      \zijsp = C_R|\log\d|^{-1} &\text{ on } \bd B(\aj, \Rmi), \\
      \zijsp = 0 &\text{ on } \bd\odj.
    \end{cases}
  \end{equation}
  This problem is radially symmetric in
  $B(\aj, \Rmi) \setminus \overline{B(\aj, \d)}$. The function
  \begin{equation}
    \zijsp = C_1 I_0(r) + C_2 K_0(r),
    \quad
    r = |x - \aj|
  \end{equation}
  with
  \begin{equation}
    C_1 \sim -|\log\d|^{-1}
    \quad\text{ and }\quad
    C_2 \sim |\log\d|^{-2}.
  \end{equation}
  satisfies \eqref{eq:zijsp} because the modified Bessel functions
  $I_0$ and $K_0$ behave as $1$ and $-\log r$, respectively, near the
  origin. Therefore
  \begin{equation}
    0
    \leq \frac{\pd\zeta_i}{\pd\nu}
    \leq \frac{\pd\zijsp}{\pd\nu}
    = \frac{C_{ij}}{\d|\log\d|^2}
    \text{ on }
    \bd\odj.
  \end{equation}
  As a result
  \begin{equation}\label{eq:approx_behavior}
    \int_{\bd\odj} \Abs{\frac{\pd \zeta_i}{\partial \nu}} ds
    \leq \frac{C}{|\log\d|^2}.
  \end{equation}
  for all $i \neq j$. Combining the estimates on the behavior of
  $\zeta_i$ on $\bd\odi$ in \eqref{eq:precise_behavior} with
  \eqref{eq:approx_behavior} and estimating the constants $\ci$ using
  the fourth equation in \eqref{eq:zeta} we find:
  \begin{align}
    \pi\d^2|c_i| + \Abs{A_i^\d}
    \geq \Abs{\int_{\bd\odi} \frac{\pd \zeta}{\partial \nu} ds}
    \nonumber &\geq \Abs{\ci\int_{\bd\odi} \frac{\pd \zeta_i}{\partial \nu} ds} - \sum_{j \neq i}\Abs{\cj\int_{\bd\odi} \frac{\pd \zeta_j}{\partial \nu} ds} \\
              &\geq |c_i| \frac{C_1}{|\log\d|} - \sum_{j \neq i}^{N} |c_j| \frac{C_2}{|\log\d|^2}
  \end{align}
  or
  \begin{equation}
    |c_i| \Par{\frac{C_1}{|\log\d|}-\pi\d^2} - \sum_{j \neq i}^{N} |c_j| \frac{C_2}{|\log\d|^2}
    \leq \Abs{A_i^\d},
  \end{equation}
  with some positive $C_1, C_2 > 0$ for all $i = 1\dots N$. The
  coefficient matrix is a small perturbation of the identity matrix,
  up to the factor $C_1|\log\d|^{-1}$. This allows us to conclude that
  \begin{equation}\label{eq:ci_estimate}
    |\ci|
    \leq |D|O(\d\log\d) + O(\d\log^2\d)
  \end{equation}
  for all $i = 1\dots N$. Let
  \begin{equation}
    \ci = \max_j |\cj|.
  \end{equation}
  Then
  \begin{equation}
    |\ci|
    \leq \Abs{A_i^\d}\Par{\frac{C_1}{|\log\d|} - \pi\d^2 - (N-1)\frac{C_2}{|\log\d|^2}}^{-1}
    \leq |D|O(\d\log\d) + O(\d\log^2\d),
  \end{equation}
  hence
  \begin{equation}
    \|\zeta\|_{L^\infty(\Od)} \leq \sum_j |\cj| \leq C_1|D|\d|\log\d| + C_2\d|\log\d|^2.
  \end{equation}
  The statement of the lemma for
  \begin{equation}
    h_3 = \eta_0 + \sum_{j=1}^{N}\dj\eta_j + \sum_{j=1}^{N}\cj\zeta_j
  \end{equation}
  then follows once we combine the estimates above.
\end{proof}

\begin{proof}[Proof of Theorem \ref{thm:S1}, continued]
  We are now able to find the asymptotics for the energy
  $l_\d(D) = GL_\d[\hdd]$:
  \begin{align}\label{eq:energy_splitted_into_h}
    l_\d(D)
    \nonumber &= GL_\d[h_1 + h_2 + h_3] \\
    \nonumber &= \frac12\int_{\Od}|\nabla h_1|^2dx + \frac12\int_{\Od}|\nabla h_2|^2dx + \frac12\int_{\Od}|\nabla h_3|^2dx \\
    \nonumber &+ \frac12\int_{\O}(h_1 - \he)^2dx + \frac12\int_{\O}h_2^2dx + \frac12\int_{\Od}h_3^2dx \\
    \nonumber &+ \int_{\Od}\Bra{\nabla(h_1 - \he)\cdot\nabla\widehat{h} + (h_1 - \he)\widehat{h}}dx
                + \int_{\Od}\Bra{\nabla h_2 \cdot \nabla h_3 + h_2 h_3}dx \\
              &+ |D|^2 O(\d^2|\log\d|^3) + O(\d^2|\log\d|^3),
  \end{align}
  where $\widehat{h} = h_2 + h_3$ and the integrals over holes $\odj$
  are the source of the error. Next, we estimate each term in
  \eqref{eq:energy_splitted_into_h}. The terms that involve $h_1$ only
  do not depend on the degrees of the hole vortices and thus they do
  not play a role in the minimization of $l_\d(D)$:
  \begin{align}
    \frac12\int_{\Od}|\nabla h_1|^2dx + \frac12\int_{\O}(h_1 - \he)^2dx
    \nonumber &= \heS\frac12\int_{\Od}|\nabla\xi_0|^2dx + \heS\frac12\int_{\O}(1 - \xi_0)^2dx \\
              &= O(|\log\d|^2).
  \end{align}
  The gradient of $h_2$ gives the main quadratic term:
  \begin{align}
    \frac12\int_{\Od}|\nabla h_2|^2dx
    \nonumber &= \frac12\sum_{j=1}^{N} \djp^2 \int_{T_j} |\nabla(\theta_j(x)\phi_j(x))|^2dx \\
    \nonumber &= \pi\sum_{j=1}^{N} \djp^2 \Bra{\int_{\d}^{R/4} |K_0(r)'|^2rdr + \int_{R/4}^{R} \Abs{\frac{d}{dr}(K_0(r)\phi(r))}^2rdr } \\
              &= \pi\sum_{j=1}^{N} \djp^2 \Bra{\int_{\d}^{R/4} \Abs{-\frac1r + O(r\log{r})}^2rdr + O(1) } \\
              &= \pi\sum_{j=1}^{N} \djp^2 |\log{\d}| + |D|^2 O(1).
  \end{align}		
  The $L^2$-norm of $h_2$ is much smaller, indeed:
  \begin{align}
    \frac12\int_{\O}h_2^2dx
    = \pi\sum_{j=1}^{N}\Par{\dj}^2\int_{0}^{R/2}|\theta_j\phi|^2rdr
    = |D|^2 O(1).
  \end{align}		
  We now estimate the integral involving $\widehat{h}$ that gives the
  linear terms in terms of the degrees. Note that, since $\hdd$ and
  $h_1$ solve the homogeneous equation $\Bra{-\Lap + I}h = 0$, then so
  does their difference $\widehat{h} = \hdd - h_1$:
  \begin{align}
    \left\langle h_1 - \he, \widehat{h} \right\rangle_{H^1(\Od)}
    \nonumber &= \int_{\Od}(h_1 - \he)\Par{-\Lap\widehat{h} + \widehat{h}}dx
                - \int_{\bd\Od}(h_1 - \he)\frac{\pd\widehat{h}}{\pd\nu}ds \\
    \nonumber &= \sum_{j=1}^{N}\int_{\bd\odj}(h_1 - \he)\frac{\pd(h_2 + h_3)}{\pd\nu}ds \\
    \nonumber &= \sum_{j=1}^{N}\int_{\bd\odj}(h_1 - \he)\Bra{\dj\Par{\frac{1}{\d} + O(\d\log\d)} + O(\log\d) + |D|O(1)}ds \\
    \nonumber &= \sum_{j=1}^{N}\dj(h_1(\aj) - \he) 2\pi\d \cdot \frac{1}{\d} + O(\d|\log\d|^2) + |D|O(\d\log\d) \\
              &= - 2\pi\sigma|\log\d|\sum_{j=1}^{N}\dj(1 - \xi_0(\aj)) + O(\d|\log\d|^2) + |D|O(\d\log\d),
  \end{align}
  where use the notation
  $\left\langle u, v \right\rangle_{H^1} = \int \Bra{\nabla u \cdot
    \nabla v + uv}dx$. The other terms in
  \eqref{eq:energy_splitted_into_h} are small and are estimated using
  integration by parts:
  \begin{align}
    \|h_3\|_{H^1(\Od)}^2
    \nonumber &= \int_{\Od}h_3\Par{-\Lap h_3 + h_3}dx
                - \int_{\bd\Od}h_3 \frac{\pd h_3}{\pd\nu}ds
                = \sum_{j=1}^{N}\int_{\bd\odj}h_3\frac{\pd h_3}{\pd\nu}ds \\
    \nonumber &= C\d\Par{C_1\d|\log\d|^2 + C_2|D|}\Par{C_1|\log\d| + C_2|D|} \\
              &= O(\d^2|\log\d|^3) + |D|^2 O(\d|\log\d|) \\
    \left\langle h_2, h_3 \right\rangle_{H^1(\Od)}
    \nonumber &= \int_{\Od} h_2 \Par{-\Lap h_3 + h_3}dx
                - \int_{\bd\Od}h_2\frac{\pd h_3}{\pd\nu}ds
                = \sum_{j=1}^{N}\int_{\bd\odj}h_2\frac{\pd h_3}{\pd\nu}ds \\
    \nonumber &= \sum_{j=1}^{N} 2\pi\delta\dj K_0(\delta) \Par{C_1|\log\d| + C_2|D|} \\
              &= |D|^2 O(\d|\log\d|^2)
  \end{align}
  Combining all of the above estimates, we obtain the asymptotic
  expansion \eqref{eq:l_expansion}.
\end{proof}
\begin{corollary}\label{_cor:min_dj}
  The leading part of the energy $l_\d(Z)$ is a sum of $N$
  one-dimensional parabolas with the vertices at
  \begin{equation}
    Z_j = \sigma(1 - \xi_0(\aj))\in\mathbb R.
  \end{equation}
  Since the degrees are integer-valued, the minimizing degrees $\dj$
  are the integers, closest to $Z_j$:
  \begin{equation}\label{eq:min_dj}
    \dj = \left\llbracket\sigma(1 - \xi_0(\aj))\right\rrbracket,
  \end{equation}
  where $\llbracket x \rrbracket$ denotes the integer nearest to $x$.
\end{corollary}

\section{Energy Decomposition}\label{sec:energy_decoposition}

Since $\Par{\udd, \Add}$ is an admissible pair for the problem
\eqref{eq:min_prob}, we can use the representation of $S^1$-valued
energy \eqref{eq:l_expansion} with $D = 0$ to obtain an upper bound
\begin{equation}\label{eq:ub}
  GL_\d^\eps \Bra{\ude, \Ade}
  \leq GL_\d^\eps \Bra{u_{\delta 0}, A_{\delta 0}}
  \leq C|\log\d|^2
\end{equation}
on the energy of the minimizer of \eqref{eq:min_prob}. In order to
obtain a matching lower energy bound, we need to localize the regions
of the domain where the magnitude of the order parameter is small. To
this end, we use the following theorem.
\begin{theorem}[Ball Construction Method
  \cite{SanSer07}]\label{thm:bcm}
  For any $\alpha \in (0, 1)$ there exists $\eps_0(\alpha) > 0$ such
  that, for any $\eps < \eps_0$, if $(u, A)$ is a configuration such
  that $GL_\d^\eps[u, A] < \eps^{\alpha - 1}$, where $\eps$ is an
  inverse of the Ginzburg-Landau parameter, the following holds.
			
  For any $1 > \rho > C\eps^{\alpha / 2}$, where $C$ is a universal
  constant, there exists a finite collection of disjoint closed balls
  $\mathfrak{B} = \{B_i = B(\bi, r_i)\}_{i \in \index}$ such that
  \begin{enumerate}
  \item $r(\mathfrak{B}) = \rho$ where
    $r(\mathfrak{B}) = \sum_{i \in \index} r(B_i)$.
  \item Letting $V = \Od \cap \cup_{i \in \index} B_i$,
    \begin{equation}
      \Cur{ x \in \Od \ \mid\ ||u(x)| - 1| \geq \eps^{\alpha/4} } \subset V.
    \end{equation}
  \item Writing $\bdi = \deg(u, \bd B_i)$, if $B_i \subset \Od$ and
    $\bdi = 0$ otherwise,
    \begin{equation}
      \frac12\int_V\Bra{ |\nabla_A u|^2 + \rho^2|\curl{A}|^2 + \frac{1}{2\eps^2}(1 - |u|^2)^2 } dx
      \geq \pi \bdsum \Par{ \log{\frac{\rho}{\bdsum\eps}} - C },
    \end{equation}
    where $\bdsum = \sum_{i \in \index}|\bdi|$ is assumed to be
    nonzero and $C$ is a universal constant.
  \item There exists a universal constant $C$ such that
    \begin{equation}
      \bdsum \leq C\frac{GL_\d^\eps[u, A]}{\alpha|\log\eps|}.
    \end{equation}
  \end{enumerate}
\end{theorem}
	
We consider now a domain with $N$ holes $\odj = B(a^j, \d)$ so that
$\Od = \O \setminus \cup_{j = 1}^{N} \odjo$. Set $\alpha = 1/2$ and
$\rho = \delta^2/2$ in the ball construction method. Assume that
$\eps$ is small enough so that $|u(x)| > 1 - \theta$ on
$\Od \cap \Par{ \cup_{i \in \index} B_i }$. The parameter $\theta$
will be chosen later, in Section \ref{sec:equality_of_degrees}.
	
\begin{lemma}\label{lem:energy_decomp}
  Let $(\ude, \Ade)$ be a minimizer of the problem
  \eqref{eq:min_prob}. Then the following energy decomposition holds:
  \begin{equation}\label{eq:energy_decomp}
    GL_\d^\eps[\ude, \Ade]
    = GL_\d[\udd, \Add] + F_\d[v, B]
    - \int_{\Od} \nabla^{\bot}\hdd \cdot \Im{\overline{v} \nabla v}\,dx
    + o(1)
  \end{equation}
  where $\ude = v \, \udd$, $\Ade = \Add + B$, $\hdd = \curl \Add$ and
  \begin{equation}
    F_\d[v, B] = \frac{1}{2}\int_{\Od}\Par{ |(\nabla - iB) v|^2
      + \frac{1}{2\eps^2}(1-|v|^2)^2}\,dx
    + \frac{1}{2}\int_\Omega(\curl B)^2\,dx.
  \end{equation}
  Here $(\udd, \Add)$ is the minimizer of the $S^1$-valued problem
  \eqref{eq:min_prob_s1_prescribed_D} with the prescribed degrees $D$.
\end{lemma}

\begin{proof}
  Using the representation \eqref{eq:S1_func_h} of Ginzburg-Landau
  functional in terms of $\hdd$, note that the pair $(\udd, \Add)$
  satisfies the following equation
  \begin{equation}
    \nabla^\bot\hdd = - \Im\,(\uddo\nabla \udd - i\Add)
  \end{equation}
  outside of the holes.  We start the proof with representing
  $GL_\d^\eps[\ude, \Ade]$ as a sum of three terms:
  \begin{equation}
    GL_\d[\ude, \Ade] = I_1 + I_2 + I_3,
  \end{equation}
  where
  \begin{align}
    I_1 = \frac{1}{2}\int_{\O}|\nabla\ude - i\Ade\ude|^2\,dx,\quad
    I_2 = \frac{1}{4\eps^2}\int_{\Od}(1-|\ude|^2)^2\,dx,\quad
    I_3 = \frac{1}{2}\int_{\O}(\curl\Ade - \he)^2\,dx.
  \end{align}
  Observe that $|\ude| = |v|$ as $\ude = v\,\udd$ and $|\udd| =
  1$. Hence we can rewrite $I_2$ as
  \begin{equation}
    I_2
    = \frac{1}{4\eps^2}\int_{\Od}(1 - |\udd|^2)^2\,dx
    = \frac{1}{4\eps^2}\int_{\Od}(1 - |v|^2)^2\,dx,
  \end{equation}
  giving us the second term in the definition of $F_\d[v, B]$. Now
  rewrite $I_3$:
  \begin{align}
    I_3
    \nonumber &= \frac{1}{2}\int_{\O}(\curl\Ade - \he)^2\,dx \\
    \label{eq:rewrt_I3} &= \frac{1}{2}\int_{\O}(\hdd - \he)^2\,dx
                          + \frac{1}{2}\int_{\O}(\curl B)^2\,dx
                          + \int_{\O}\curl{B}\cdot(\hdd - \he)\,dx
  \end{align}
  Here, the first term is a part of $GL_\d[\udd, \Add]$ and the second
  term is a part of $F_\d[v, B]$. The last term will eventually cancel
  with a component of $I_1$. To this end,
  \begin{align} |\nabla\ude - i \Ade\ude|^2
    \nonumber &= \Abs{ v \Par{\nabla\udd - i\Add\udd} + \udd \Par{\nabla v - i B v} }^2 \\
    \nonumber &= |v|^2 |\nabla\udd - i\Add\udd|^2 + |\udd|^2|\nabla v - i B v|^2 \\
    \nonumber &+ 2\Re\Par{ \uddo \Par{\nabla\udd - i\Add\udd} \cdot v \Par{\nabla\overline{v} + i B \overline{v}} } \\
    \nonumber &= |\nabla v - i B v|^2 + |v|^2|\nabla\hdd|^2 \\
              &+ 2|v|^2 \nabla^\bot\hdd \cdot B - 2\nabla^\bot\hdd
                \cdot \Im(\overline{v}\nabla v)\label{eq:80}
  \end{align}
  The first term in \eqref{eq:80} contributes to $F_\d[v, B]$. The last term is included
  in the right hand side of the decomposition. The sum of two other
  terms has the form $|v|^2 \cdot R(x)$, where
  \begin{equation*}
    R(x) = |\nabla\hdd|^2 + 2\nabla^\bot\hdd \cdot B
  \end{equation*}
  Now add and subtract $\frac12\int_{\Od} R(x)\,dx$ to the energy
  $GL_\d[\ude, \Ade]$. The first term
  $\frac12\int_{\Od} |\nabla\hdd|^2 \,dx$ is a part of
  $GL_\d[\udd, \Add]$. Using integration by parts we prove that the
  second term $\int_{\Od} \nabla^\bot\hdd \cdot B \,dx$ indeed cancels
  with the last term in the representation \eqref{eq:rewrt_I3} of
  $I_3$ as alluded to above:
  \begin{align}\label{int_oe}
    \int_{\Od} \nabla^\bot \hdd \cdot B\,dx
    \nonumber &= \int_{\Od} \nabla^\bot(\hdd - \he) \cdot B \,dx \\
    \nonumber &= \int_{\bd\Od} (h_\d - \he) B \cdot \tau \,dS
                - \int_{\Od} (\hdd - \he) \nabla^\bot \cdot B \,dx \\
    \nonumber &= -\sum_{j=1}^{N} \left.(\hdd - \he)\right|_{\bd B(\aj, \d)} \int_{\bd B(\aj, \d)} B \cdot \tau \,dS
                - \int_{\Od} (\hdd - \he) \,\curl B \,dx \\
    \nonumber &= -\sum_{j=1}^{N} \left.(\hdd - \he)\right|_{\bd B(\aj, \d)} \int_{B(\aj, \d)} \curl B \,dS
                - \int_{\Od} (\hdd - \he) \,\curl B \,dx \\
              &= - \int_{\O} (\hdd - \he) \,\curl B \,dx.
  \end{align}
  Here we used the facts that $\hdd = \he$ on the boundary $\bd\O$ and
  $\hdd = \const$ in $B(\aj, \d)$ that follow from the equation for
  $\hdd$.
		
  Adding up the results above gives:
  \begin{align}
    GL_\d^\eps[\ude, \Ade]
    \nonumber &= GL_\d[\udd, \Add]
                + F_\d[v, B] \\
              &- \int_{\Od} \nabla^\bot \hdd \cdot \Im{\overline{v} \nabla v}\,dx
                + \int_{\Od} (1 - |v|^2) R(x) \,dx + o(1)
  \end{align}
  The remaining task is to show that
  \begin{equation*}
    I = \int_{\Od} (1 - |v|^2) R(x) \,dx
  \end{equation*}
  goes to zero as $\d\to 0$. H\"older's inequality implies that
  \begin{equation}
    |I| \leq \|1 - |v|^2\|_{L^2(\Od)} \cdot \Par{ 2\|\nabla\hdd\|_{L^4(\Od)}^2 + \|B\|_{L^4(\Od)}^2 }.
  \end{equation}
  The first multiplier in this expression is less then $M\eps|\log\d|$
  when $\d \to 0$ because of the a priori estimate on the
  energy. Using the relation between $\eps$ and $\d$
  \begin{equation}
    |\log\eps| \gg |\log\d|,
  \end{equation}
  we show that $\eps$ is sufficiently small to compensate for the
  growth of the other terms.
		
  The function $\hdd$ is described in Theorem \ref{thm:S1} and because
  of Lemma \ref{lem:grad_est} it satisfies the estimate
  \begin{equation}
    \|\nabla\hdd\|_{L^4(\Od)}^2 \leq \frac{C|\log\d|^2}{\d^2}.
  \end{equation}
  In order to estimate $\|B\|_{L_4(\Od)}$, recall that
  $\div{\Ade} = 0$ due to the gauge invariance. Then by the
  Poincar\'e's lemma $\Ade$ has a potential, i.e. there exists $\Pde$
  such that $\nabla^\bot\Pde = \Ade$. Substituting this into
  $\hde = \curl\Ade$, we obtain the equality $\Lap\Pde = \hde$. The
  function $\Pde$ is a potential so we are able to make it zero on the
  boundary $\bd\O$. From the theory of elliptic operators and the a
  priori energy estimate, we obtain
  \begin{equation}
    \|\Pde\|^2_{H^2(\O)} \leq
    \|\hde\|^2_{L_2(\O)} \leq C|\log\d|^2.
  \end{equation}
  Since the embedding $H^1(\O) \subset L^4(\O)$ is continuous we have
  \begin{equation*}
    \|\Ade\|_{L_4(\Od)} \leq C\|\Pde\|_{H^2(\O)} \leq C|\log\d|.
  \end{equation*}
  The same estimate holds for $\Add$. Using the decomposition
  $\Ade = B + \Add$ we obtain this estimate for $B$:
  \begin{equation*}
    \|B\|_{L_4(\Od)} \leq C|\log\d|
  \end{equation*}
		
  Combining all estimates obtained in this section, we conclude that
  \begin{equation*}
    |I| \leq C\eps|\log\delta| \Par{\frac{|\log\d|^2}{\d^2} + |\log\d|^2}.
  \end{equation*}
  The condition $|\log\eps| \gg |\log\d|$ implies that $\eps$ is much
  smaller than any power of $\delta$, therefore $I$ goes to zero as
  $\d \to 0$ that completes the proof.
\end{proof}

\section{Absence of Bulk Vortices}\label{sec:absence_of_bulk_vortices}

In this section we further analyze the energy decomposition
\eqref{eq:energy_decomp}. The energy of the unconstrained solution is
minimal, hence
\begin{align}\label{eq:comp_min}
  GL_\d^\eps[\ude, \Ade] \leq GL_\d[\udd, \Add],
\end{align}
and using \eqref{eq:energy_decomp} we have
\begin{equation}\label{eq:decomp_ineq}
  F_\d[v, B]
  \leq \int_\Od \nabla^\bot \hdd \cdot \Im{\overline{v}\nabla v}\,dx
  + o(1).
\end{equation}
First, we derive an upper bound for the integral term in
\eqref{eq:decomp_ineq} and thus for the energy $F_\d$. We start with a
simple fact that will also be used later on.
\begin{proposition}
  Given a sufficiently smooth domain $S\subset\mathbb R^2$ and any
  $R \in L^2(S,\mathbb R)$, $P \in H^1(S,\mathbb R^2)$, and
  $v \in H^1(S,\mathbb C)$ such that $|v|\leq1$ a.e. $x\in S$, we have
  that
  \begin{align}
    \left| \int_{S} R(x) \cdot \Im{\overline{v} \nabla v}\,dx \right|
    \nonumber &\leq \left| \int_{S} R(x) \cdot \left( \Im{\overline{v} (\nabla - iP)v} + P|v|^2 \right)\,dx \right| \\
    \label{eq:aux_est_1} &\leq \|R\|_{L^2(S)} \cdot \left( \|(\nabla - iP)v\|_{L^2(S)} + \|P\|_{L^2(S)} \right)
  \end{align}
\end{proposition}
	
We are now in the position to state and prove
\begin{lemma}\label{lem:F_upper_bound}
  The following estimates hold:
  \begin{align}
    F_\d[v, B] &\leq |\log\d|^2, \\
    \label{eq:h_int_upp_bound} \Abs{\int_\Od \nabla^\bot \hdd \cdot \Im{\overline{v}\nabla v}\,dx} &\leq |\log\d|^2.
  \end{align}
\end{lemma}
\begin{proof}
  Use \eqref{eq:aux_est_1} and Poincar\'e inequality to estimate the
  integral term in \eqref{eq:decomp_ineq}:
  \begin{align}\label{eq:h_int_upp_bound_proof}
    \left| \int_{\Od} \nabla^\bot \hdd \cdot \Im{\overline{v} \nabla v}\,dx \right|
    \nonumber &\leq \|\nabla \hdd\|_{L^2(\Od)} \cdot \left( \|(\nabla - iB)v \|_{L^2(\Od)}
                + C_{\O}\|\curl{B}\|_{L^2(\O)} \right) \\
    \nonumber &\leq \frac{1}{2\alpha}\|\nabla \hdd\|_{L^2(\Od)}^2
                + \frac{\alpha}{2}\left( \|(\nabla - iB)v \|_{L^2(\Od)}^2
                + C_{\O}^2\|\curl{B}\|_{L^2(\Omega)}^2 \right) \\
              &\leq O(|\log{\d}|^2) + \frac12 F_\d[v, B]
  \end{align}
  where $\alpha = \min(1, C_{\O}^{-2})$. Here we have used the
  standard fact that $|\ude|\leq1$ and, therefore, $|v|\leq1$
  a.e. $x\in\Od$.
		
  Combining the inequality \eqref{eq:h_int_upp_bound_proof} with
  \eqref{eq:decomp_ineq} gives
  \begin{equation}\label{eq:f_upp_bound}
    F_\d[v, B]
    \leq O(|\log\d|^2).
  \end{equation}
  The estimates \eqref{eq:h_int_upp_bound_proof} and
  \eqref{eq:f_upp_bound} imply \eqref{eq:h_int_upp_bound}.
\end{proof}

The bound \eqref{eq:f_upp_bound} allows us to apply the ball
construction method to $F_\d$. Theorem \ref{thm:bcm} gives the
following lower bound on the energy inside ``bad'' disks:
\begin{equation}\label{eq:f_o}
  F_\d[v, B; B_i] \geq \pi |d_i| \Par{ \log{\frac{\d^2}{|d_i|\eps}} - C } \text{ for every } i \in \index.
\end{equation}
Here $F_\d[v, B; B_i]$ is the energy $F_\d[v, B]$ where first two
integrals are taken over the domain $B_i = B(\bi, r_i)$. To continue
working with \eqref{eq:decomp_ineq} we prove the following lemma.
\begin{lemma}\label{lem:int_repres}
  The following representation holds:
  \begin{equation}\label{eq:int_representation}
    \int_{\Od}\nabla^\bot\hdd \cdot \Im{\overline{v} \nabla v}\,dx
    = 2\pi \sum_{i \in \index_1} (\he - \hdd(\bi))\bdi + 2\pi \sum_{j=1}^N \dvj(\he - \Hrj) + O(1)
  \end{equation}
  where $\dvj = \deg(v, \grj) = \dedj - \dj$, the circular curves
  $\grj = \bd B(\aj, R)$ enclose $\odj$ with $R = \d + O(\d^2)$, the
  quantities $\Hrj = \dj K_0(R) + \he\xi_0(\aj)$, and $\index_1$
  includes only the balls that are proper subsets of
  $\Od \setminus \cup_{j=1}^N \odjo$ and do not intersect the boundary
  $\bd\Od$.
\end{lemma}
\begin{proof}
  We divide the domain $\Od$ into three disjoint parts:
  \begin{equation}
    \Od = S \cup V \cup G,
  \end{equation}
  where $S = \cup_{j=1}^N S_j$ consists of the annuli between
  $\bd\odj$ and $\grj$, the set
  $V = \Bra{(\cup_{i \in \index} B_i) \setminus S} \bigcap \Od$
  consists of the ``bad" disks, and $G$ corresponds to the remainder
  of the set $\Od$.

  Consider the subdomains $S$, $V$, and $G$ separately. The balls
  $B_i$---as well as stripes $S_j$---are very small so that
  \begin{align}
    \int_{V \cup S}\nabla^\bot\hdd \cdot \Im{\overline{v} \nabla v}\,dx
    \nonumber &\leq \meas(V \cup S)^{1/4} \cdot \|\nabla^{\bot}\hdd\|_{L^4(V \cup S)} \\
    \nonumber &\cdot \Par{ \|(\nabla - iB)v \|_{L^2(V \cup S)} + \|B\|_{L^2(V \cup S)} } \\
              &\leq C\d^{3/4} \cdot |\log\d| \cdot |\log\d| = o(1).
  \end{align}
  Introduce the function $w = v / |v|$. Then
  \begin{align}
    \int_{G}\nabla^{\bot}\hdd \cdot \Im{\overline{v} \nabla v}\,dx
    \nonumber &= \int_{G}\nabla^{\bot}\hdd \cdot \Im{\overline{w} \nabla w}\,dx
                + \int_{G}\nabla^{\bot}\hdd \cdot \left( \Im{\overline{v} \nabla v} - \Im{\overline{w} \nabla w} \right)\,dx \\
              &= I_1 + I_2.
  \end{align}

  To estimate the second integral, use the following:
  \begin{align}
    \Im{\overline{v} \nabla v} - \Im{\overline{w} \nabla w}
    \nonumber &= \Im\left( \overline{w}|v|(w\nabla|v| + |v|\nabla w) - \overline{w} \nabla w \right) \\
              &= \Im\left( |v|\nabla|v| + (|v|^2 - 1)\overline{w}\nabla w
                \right) = (|v|^2 - 1)\Im{\overline{w}\nabla w}
  \end{align}
  and
  \begin{equation}
    |\nabla v|^2 = |v|^2|\nabla w|^2 + |\nabla |v|| \geq (1-\theta)^2|\nabla w|^2 \geq \frac14|\nabla w|^2.
  \end{equation}
  since by Theorem \ref{thm:bcm} we have $|v| \geq 1 - \theta$ outside
  $B_i$. The function $v$ admits the same estimate as $\ude$. Add and
  subtract $iBv$ to get
  \begin{align}
    \frac12\|\nabla v\|_{L^2(G)}^2
    \nonumber &= \frac12\int_{G} |\nabla v|^2\,dx
                \leq \int_{G} \left( |(\nabla - iB) v|^2 + |v|^2|B|^2\right)\,dx \\
              &\leq \int_{\Od} |(\nabla - iB) v|^2\,dx
                + C_{\O}\int_{\O}|\curl{B}|^2\,dx \leq C|\log{\d}|^2.
  \end{align}
  This leads to the following estimate:
  \begin{align}
    |I_2|
    \nonumber &\leq \int_{G}\nabla^{\bot}\hdd \cdot (|v|^2 - 1)\Im{\overline{w}\nabla w}\,dx \\
    \nonumber &\leq \|\nabla^{\bot}\hdd\|_{L^{\infty}(G)} \cdot \int_{G}(|v|^2 - 1) \cdot |\nabla w|\,dx \\
    \nonumber &\leq C\d^{-1} \cdot \int_{G}(|v|^2 - 1) \cdot 2|\nabla v|\,dx \\
    \nonumber &\leq C\d^{-1} \cdot \||v|^2 - 1\|_{L^2(G)} \cdot \|\nabla v\|_{L^2(G)} \\
              &\leq C\d^{-1} \cdot \eps |\log{\d}| \cdot |\log{\d}|
                = o(1)
  \end{align}
  due to \eqref{eq:cond_eps}.
			
  Now rewrite the integral $I_1$. Integrating by parts, we obtain:
  \begin{align}
    I_1
    \nonumber &= \int_{G}\nabla^{\bot}(\hdd - \he)\cdot\Im{\overline{w}\nabla w}\,dx
                = -\int_{G}(\hdd - \he)\nabla^{\bot} \cdot \Im{\overline{w}\nabla w}\,dx \\
    \nonumber &+\int_{\bd\O}(\hdd - \he) \Im{\overline{w}\nabla w} \cdot \tau \,ds
                -\int_{\bd V}(\hdd - \he) \Im{\overline{w}\nabla w} \cdot \tau \,ds \\ \nonumber &
                                                                                                   -\int_{\cup_j\gamma_j}(\hdd - \he) \Im{\overline{w}\nabla w} \cdot \tau \,ds \\
              &= -\sum_{i \in \index} I_{1i}
                - \sum_{j=1}^{N}\int_{\grj}(\hdd - \he)\, \Im{\overline{w}\nabla w} \cdot \tau \,ds
  \end{align}
  where
  $I_{1i} = \int_{\bd V_i}(\hdd - \he) \,\Im{\overline{w}\nabla w}
  \cdot \tau \,ds$ and $V_i = B_i \cap \Od$. The term
  $\nabla^{\bot}\cdot\Im{\overline{w}\nabla w} = \curl\nabla\Phi = 0$,
  where $\Phi$ is a phase of $w$, disappears.
		
  Since the curves $\grj$ are small, we can approximate $\hdd$ by a
  constant $\Hrj$ to conclude that
  \begin{align}
    \int_{\grj}(\hdd - \he)\, \Im{\overline{w}\nabla w} \cdot \tau \,ds
    \nonumber &= 2\pi\dvj(\Hrj - \he)
                + \int_{\grj}(\hdd - \Hrj)\, \Im{\overline{w}\nabla w} \cdot \tau \,ds.
  \end{align}
  Set $\Hrj = \he\xi_0(\aj) + \dj K_0(R)$. Using the decomposition
  \eqref{eq:magn_field_decomposition} of $\hdd$, we get
  \begin{equation}
    |\hdd(x) - \Hrj|
    \leq \he|\xi_0(x) - \xi_0(\aj)| + |h_3(x)|
    \leq C_1\d|\log\d|^2 + C_2|D|
  \end{equation}
  for $x \in \grj$. This yields
  \begin{equation}
    \Abs{\int_{\grj}(\hdd - \Hrj) \,\Im{\overline{w}\nabla w} \cdot \tau \,ds}
    \leq \Par{C_1\d|\log\d|^2 + C_2|D|} \cdot \dvj = O(1).
  \end{equation}
  As a result we estimate that
  \begin{equation}
    I_1
    = -\sum_{i \in \index} I_{1i}
    - \sum_{j=1}^{N}2\pi\dvj(\Hrj - \he)
    + O(1).
  \end{equation}

  We now consider two cases. First, suppose that the set
  $\index_1 \subset \index$ is such that $B_i \subset \Od \setminus S$
  for $i \in \index_1$. We estimate the integrals $I_{1i}$ in a
  similar way as we did for the hole vortices. Approximate $\hdd(x)$
  by a constant value in the center of $B_i$:
  \begin{equation}
    I_{1i}
    = \int_{\bd V_i}(\hdd - \hdd(\bi)) \,\Im{\overline{w}\nabla w} \cdot \tau \,ds
    + \int_{\bd V_i}(\hdd(\bi) - \he) \,\Im{\overline{w}\nabla w} \cdot \tau \,ds
    = J_{1i} + J_{2i}
  \end{equation}
  Second integral directly gives the degree $d_i$ of the possible bulk
  vortex:
  \begin{equation}
    J_{2i} = 2\pi d_i (\hdd(\bi) - \he).
  \end{equation}
			
  To estimate $J_{1i}$ we introduce the subdomains
  $U_i = V_i \cap \{x\,\mid\,|v(x)| \leq 1/2\}$ so that their
  boundaries are the level sets of $v$. We add and subtract the
  integral over $\bd U_i$:
  \begin{equation}
    \sum_{i \in \index} J_{1i} = J_1 + J_2,
  \end{equation}
  where
  \begin{align}
    J_1 &= \int_{\cup_{i \in \index_1} \bd U_i}(\hdd - \hdd(\bi))
          \Im{\overline{w}\nabla w} \cdot \tau ds, \\
    \nonumber J_2 &= \int_{\cup_{i \in \index_1} \bd V_i}(\hdd - \hdd(\bi))
                    \Im{\overline{w}\nabla w} \cdot \tau ds
                    - \int_{\cup_{i \in \index_1}\bd U_i}(\hdd - \hdd(\bi))
                    \Im{\overline{w}\nabla w} \cdot \tau ds \\
        &= \int_{\cup_{i \in \index_1}(V_i \setminus U_i)} \nabla^{\bot}\cdot\Bra{(\hdd - \hdd(\bi))\Im{\overline{w}\nabla w}}dx
          = \int_{\cup_{i \in \index_1}(V_i \setminus U_i)} \nabla^{\bot}\hdd \cdot \Im{\overline{w}\nabla w}dx,
  \end{align}
  since $\nabla^{\bot}\cdot\Im{\overline{w}\nabla w} = 0$. The term
  $J_2$ is small:
  \begin{equation}\label{eq:est_I_i}
    |J_2|
    \leq \meas(\mathfrak{B})^{1/2} \cdot \|\nabla^{\bot}\hdd\|_{L^\infty(\mathfrak{B})} \cdot 2\|\nabla v\|_{L^2(\mathfrak{B})}
    \leq O(\d^2) \cdot O\Par{\frac1\delta} \cdot O(|\log{\d}|)
    = o(1).
  \end{equation}
  To estimate $J_1$, note, that $|v| = 1/2$ on $\bd U_i$ so that
  $\nabla w \cdot \tau = 2\nabla v \cdot \tau$ on $\bd U_i$ and:
  \begin{align}
    J_1 
    \nonumber &= \int_{\cup_{i \in \index_1} \bd U_i}(\hdd - \hdd(\bi)) \Im{\overline{w}\nabla w} \cdot \tau \,ds
                = 4\int_{\cup_{i \in \index_1} \bd U_i}(\hdd - \hdd(\bi)) \Im{\overline{v}\nabla v} \cdot \tau \,ds \\
    \nonumber &= 4\int_{\cup_{i \in \index_1} U_i}\nabla^{\bot}\hdd\cdot\Im{\overline{v}\nabla v}\,dx +
                4\int_{\cup_{i \in \index_1} U_i}(\hdd - \hdd(\bi))\Im(\nabla^{\bot}\overline{v}\cdot\nabla v)\,dx =
                L_1 + L_2.
  \end{align}
  The first integral $L_1$ admits the same estimate as in
  \eqref{eq:est_I_i}. To estimate $L_2$ note that
  \begin{equation}
    |\Im(\nabla^{\bot}\overline{v}\cdot\nabla v)|
    \leq |\nabla^{\bot}\overline{v}|\cdot|\nabla v|
    = |\nabla v|^2.
  \end{equation}
  Then
  \begin{align}
    |L_2|
    \nonumber &\leq 4\sum_{i \in \index_1}\|\hdd - \hdd(\bi)\|_{L^{\infty}(U_i)} \cdot \|\nabla v\|_{L^2(\O)}^2 \\
    \label{eq:L_2_est} &\leq 4\sum_{i \in \index_1}\|\nabla\hdd\|_{L^{\infty}(U_i)} \cdot r_i \cdot |\log{\d}|^2
                         \leq O\Par{\frac1\delta} \cdot \d^2 \cdot |\log{\d}|^2 = o(1).
  \end{align}
  Thus all integrals $L_1$, $L_2$, and therefore $J_1$, $J_2$, and
  $J_{1i}$ are small. The only ingredient left to consider is the set
  $\index_2$ consisting of the balls that intersect the boundary
  $\bd\O$. Here the estimates are very similar to those on the balls
  from $\index_1$ if we recall the boundary condition $\hdd = \he$ on
  $\bd\O$:
  \begin{align}
    \sum\limits_{i \in \index_2} I_{1i}
    \nonumber &= \int_{\cup_{i \in \index_2}\bd V_i}(\hdd - \he) \Im{\overline{w}\nabla w} \cdot \tau \,ds \\
    \nonumber &= 4\int_{\cup_{i \in \index_2}\bd U_i}(\hdd - \he) \Im{\overline{v}\nabla v} \cdot \tau \,ds
                + \int_{\cup_{i \in \index_2}(V_i \setminus U_i)}\nabla^{\bot}\hdd \cdot \Im{\overline{w}\nabla w}\,dx \\
    \nonumber &= 4\int_{\cup_{i \in \index_2}U_i}\nabla^{\bot}(\hdd - \he) \cdot \Im{\overline{v}\nabla v} \,dx
                + 4\int_{\cup_{i \in \index_2} U_i}(\hdd - \he) \Im(\nabla^{\bot}\overline{v}\cdot\nabla v)\,dx + o(1) \\
              &= o(1).
  \end{align}
  The external magnetic field here plays the same role as $\hdd(\bi)$
  in \eqref{eq:L_2_est}, that is:
  \begin{equation}
    |\hdd(x) - \he|
    \leq \|\nabla \hdd\|_{L^{\infty}(\O)} \cdot 2r_i
    \leq O(\d)
  \end{equation}
  in $B_i$ for $B_i \cap \bd\O \neq \emptyset$ because $\hdd = \he$ on
  $\bd\O$.
            
  Combining the estimates we obtain
  \begin{equation}
    \sum_{i \in \index} I_{1i} = \sum_{i \in \index_1} 2\pi d_i (\hdd(\bi) - \he) + o(1),
  \end{equation}
  thus concluding the proof.
\end{proof}
	
Putting together \eqref{eq:decomp_ineq}, \eqref{eq:f_o}, and
\eqref{eq:int_representation} we get
\begin{equation}\label{eq:decomp_ineq_2}
  F_\d[v, B; G]
  + \pi d \left( \log{\frac{\d^2}{d\eps}} - C \right)
  \leq 2\pi \sum_{i \in \index_1} (\he - \hdd(\bi))\bdi
  + 2\pi \sum_{j=1}^N \dvj(\he - \Hrj)
  + O(1),
\end{equation}
where $d = \sum_{i \in \index} |d_i|$ as before. This inequality holds
under the assumption that $d$ is nonzero. If, on the other hand, $d$
equals zero, the term
$\pi d \left( \log{\frac{\d^2}{d\eps}} - C \right)$ should be dropped.
	
In the following lemma we obtain the lower bound for $F_\d$ that
allows us to show that there are no bulk vortices, i.e, $d_i = 0$.
\begin{lemma}\label{lem:absence_of_bulk_vortices}
  There exists a $\d_0>0$ such that, for any $\d \leq \d_0$, there are
  no bulk vortices inside the domain $\Omega \setminus
  \overline{S}$. Moreover, there exist an $\alpha > 1$ and an
  $\d \ll R' \ll 1$ such that the following inequality holds:
  \begin{equation}\label{eq:quad_ineq}
    \sum_{j=1}^N
    \Bra{ \pi(1 - \theta)^2(|\log{\d}| - |\log{R'}| + O(\d))\dvjp^2
      - 2\pi \dvj (\he - \Hrj) }
    \leq O(1).
  \end{equation}
\end{lemma}
\begin{proof}
  Fix $\alpha > 1$ and consider two cases:
  \begin{enumerate}
  \item $\sum_{j=1}^N|\dvj| \leq \alpha\sum_{i \in \index} |d_i|$. The
    leading term in \eqref{eq:decomp_ineq_2} is $\pi d |\log{\eps}|$
    on the left hand side and it cannot be bounded by the right hand
    side if $d \neq 0$ because the leading term there is of order
    $d \cdot O(|\log{\d}|)$. Therefore $d = 0$, and there are no bulk
    vortices and all $\dv = 0$.
  \item $\sum_{j=1}^N|\dvj| > \alpha\sum_{i \in \index} |d_i|$. We
    need an additional lower bound on the energy $F_\d[v, B; G]$.
  \end{enumerate}
		    
  To estimate $F_\d[v, B; G]$, we integrate over circles
  $\grj = \bd B(\aj, r)$ around the holes $\odj$ with $r > R$. If
  $|u| \neq 0$ on $\grj$ for some $r > R$, we can define the degree on
  $\grj$ via
  \begin{equation}
    \drj = \deg(u, \grj) = \deg(v, \grj)
  \end{equation}
  Denote
  \begin{equation}
    \mathfrak{R} = \{r \in (R, R_{max})\ :\ |u| > 1 - \theta
    \text{ on } \grj
    \text{ for all } j=1\dots N
    \},
  \end{equation}
  where $\theta$ is specified in the Ball Construction Method and
  $R_{max}$ plays the same role as in Lemma 1, i.e., it is the maximal
  radius $r$ such that $B(\aj, r)$ are disjoint and do not intersect
  $\bd\O$. The total degree on $\bd\O$ is the sum of the degrees of
  all vortices. Since $\drj = \dvj$ by definition of $\dvj$, we have
  \begin{equation}\label{eq:drj_vs_dvj}
    \sum_{j=1}^{N} |\drj|
    \geq \sum_{j=1}^{N} |\dvj| - \sum_{i \in \index} |d_i|
    \geq \frac{\alpha - 1}{\alpha}\sum_{j=1}^{N} |\dvj|.
  \end{equation}		    
  Using the definition of the degree and the Divergence Theorem for
  $r \in \mathfrak{R}$ we get
  \begin{equation}
    2\pi\drj - \int_{\Brj} \curl B\,dx
    = \int_{\grj} \nabla \Phi \cdot \tau - B \cdot \tau dS
    = \int_{\grj} (\nabla\Phi - B)\cdot\tau dS
  \end{equation}
  or
  \begin{equation}\label{eq:D_r}
    2\pi\drj
    = \int_{\grj} (\nabla\Phi - B)\cdot\tau dS + \int_{\Brj} \curl B\,dx
    = I_1(r) + I_2(r)
  \end{equation}
  for any $j = 1\dots N$. Here $\Brj = B(\aj, r)$ and
  $v = |v|e^{i\Phi}$. The following estimates
  \begin{align}
    I_1^2 &\leq \meas(\grj) \int_{\grj} |\nabla\Phi - B|^2 dS
            \leq 2\pi r \int_{\grj} \frac{|(\nabla - iB)v|^2}{|v|^2} dS
            \leq \frac{2\pi r}{(1 - \theta)^2}\int_{\grj} |(\nabla - iB)v|^2 dS, \\
    I_2^2 &\leq \meas(\Brj) \int_{\Brj} |\curl B|^2 dx \leq C_1 |\log{\d}|^2 r^2,
  \end{align}
  hold since $|v| > 1 - \theta$ by the Ball Construction
  Method. Further
  \begin{align}
    4\pi^2 \drjp^2
    \nonumber &= (I_1(r) + I_2(r))^2 \\
              &\label{eq:est_D_r} \leq \frac{2\pi r}{(1 - \theta)^2}\int_{\grj} |(\nabla - iB)v|^2 dS
                + 2C_1 |\log{\d}|^2 r^2 \cdot I_1
                + C_1|\log{\d}|^2 r^2
  \end{align}
  for $r \in \mathfrak{R}$. Now, divide both sides of
  \eqref{eq:est_D_r} by $r$ and integrate outside of the ``bad" disks
  from $R$ to $R'$
  \begin{align}
    4\pi^2\int_{(R, R') \cap \mathfrak{R}} \frac{\drjp^2}{r}\,dr
    \nonumber &\leq \frac{2\pi}{(1 - \theta)^2} \int_{(R, R') \cap \mathfrak{R}}
                \int_{\grj} |(\nabla - iB)v|^2 dS dr \\
    \nonumber &+ 2C_1 |\log{\d}|^2 \cdot \int_{(R, R') \cap \mathfrak{R}} I_1 r dr
                + C_1|\log{\d}|^2 \left.\frac{r^2}{2}\right|_{R}^{R'} \\
    \nonumber &\leq \frac{4\pi}{(1 - \theta)^2} F_\d[v, B; B_{R'}^j] + \frac{C_1}{2}|\log{\d}|^2 R'^2 \\
    \nonumber &+ 2C_1|\log{\d}|^2 \cdot R'\cdot \sqrt{\pi R'^2} \cdot
                \left( \int_{(R, R') \cap \mathfrak{R}} \int_{\gamma_r} \frac{|(\nabla - iB)v|^2}{|v|^2} dS dr \right)^{1/2} \\
    \label{eq:up_est} &\leq \frac{4\pi}{(1 - \theta)^2} F_\d[v, B; K^j]
                        + \frac{C_1}{2}|\log{\d}|^2 R'^2 + \frac{C_3}{1 - \theta} |\log{\d}|^3 R'^2,
  \end{align}
  since $|v| > 1 - \theta$ by the definition of $\mathfrak{R}$. Here
  $R' \ll R_{max}$ that will be prescribed later on and $K^j$ is a
  union of concentric rings around $j$th hole:
  \begin{equation}
    K^j
    = \bigcup_{r \in (R, R') \cap \mathfrak{R}} \grj
    = \bigcup_{r \in (R, R') \cap \mathfrak{R}} \bd B(\aj, r).
  \end{equation}
  Notice that all $K^j$ are disjoint since $R' \ll R_{max}$ and
  $K^j \subset G$ for all $j = 1\dots N$.

  In order to obtain the lower bound for $F_\d$ we divide both sides
  in \eqref{eq:up_est} by $4\pi/(1 - \theta)^2$:
  \begin{equation}\label{eq:up_est_2}
    \pi(1 - \theta)^2\int_{(R, R') \cap \mathfrak{R}} \frac{\drjp^2}{r}\,dr
    \leq F_\d[v, B; K^j]
    + \frac{C_1(1 - \theta)^2}{8\pi}|\log{\d}|^2 R'^2 + \frac{C_3(1 - \theta)}{4\pi} |\log{\d}|^3 R'^2.
  \end{equation}
  We can choose
  \begin{equation}\label{eq:cond_zeta}
    R' = C\zeta^{1/2}|\log{\d}|^{-2} \gg R
  \end{equation}
  and an appropriate constant $C$ such that for
  $\zeta = |\log{\d}|^{-1} = o(1)$ the sum of last two terms in
  \eqref{eq:up_est_2} is less than $\zeta$ for small $\d$. Notice,
  that $\meas((R, R') \setminus \mathfrak{R}) < \d^2$ by the Ball
  Construction Method and $R \leq \d + \d^2$. Therefore
  \begin{align}\label{eq:low_est}
    \sum_{j=1}^N\int_{(R, R') \cap \mathfrak{R}} \frac{\drjp^2}{r}\,dr
    \nonumber &\geq \left. \frac{(\alpha - 1)^2}{\alpha^2}\frac1N\sum_{j=1}^{N}|\dvj|^2 \log{r} \right|_{\d + 2\d^2}^{R'} \\
              &\geq \frac{(\alpha - 1)^2}{\alpha^2}\frac1N\sum_{j=1}^{N}|\dvj|^2 (|\log{\d}| - |\log{R'}| + O(\d)).
  \end{align}
  Thus we can combine \eqref{eq:up_est_2} and \eqref{eq:low_est} to
  express the lower bound for $F_\d[v, B; G]$ in terms of the
  additional degrees $\dvj$:
  \begin{align}\label{eq:F_lower_bound_1}
    F_\d[v, B; G]
    \nonumber &\geq \sum_{j=1}^{N} F_\d[v, B; K^j] \\
              &\geq \pi(1 - \theta)^2\frac{(\alpha - 1)^2}{\alpha^2}\frac1N\sum_{j=1}^{N}|\dvj|^2 (|\log{\d}| - |\log{R'}| + O(\d)) - \zeta.
  \end{align}
  Substituting $\zeta = |\log{\d}|^{-1}$ and combining
  \eqref{eq:F_lower_bound_1} with \eqref{eq:decomp_ineq_2}, we get
  \begin{gather}
    \sum_{j=1}^{N}\Par{ \frac1N\pi(1 - \theta)^2\frac{(\alpha -
        1)^2}{\alpha^2}(|\log{\d}| - |\log{R'}| + O(\d))\dvjp^2
      \nonumber - 2\pi (\he - \Hrj)\dvj } \\
    \label{eq:decomp_ineq_3} \leq -\pi \sum_{i \in \index_1}|\bdi|
    \left( |\log{\eps}| - 2|\log{\d}| + |\log{d}| - C \right) + 2\pi
    \sum_{i \in \index_1} (\he - \hdd(\bi))\bdi + O(1).
  \end{gather}
  Compare the order of the leading terms in \eqref{eq:decomp_ineq_3}:
  \begin{equation}\label{eq:F_lower_bound_order}
    \sum_{j=1}^{N}\Par{ A|\log\d|\dvjp^2 - O(|\log\d|)\dvj } \leq -d|\log\eps| + O(1)
  \end{equation}
  with $A > 0$. The left hand side of \eqref{eq:F_lower_bound_order}
  is a sum of quadratic functions in $\dvj$ with positive leading
  coefficients:
  \begin{equation}
    q_j(\dvj) = A|\log\d|\dvjp^2 - O(|\log\d|)\dvj.
  \end{equation}
  The values of parabolas $q_j$ are bounded from below by the values
  at their vertices
  \begin{equation}
    t^j = \frac{O(|\log\d|)}{2A|\log\d|} = O(1),
  \end{equation}
  that are themselves bounded. Therefore
  \begin{equation}\label{eq:eps_del_main_contr}
    -d|\log\eps| + o(1)
    \geq \sum_{j=1}^{N} q_j(\dvj)
    \geq \sum_{j=1}^{N} q_j(t_j)
    = O(|\log\d|).
  \end{equation}
  Since $|\log\eps| \gg |\log\d|$, the inequality
  \eqref{eq:eps_del_main_contr} can hold only if $d = 0$, i.e. there
  are no bulk vortices. This, in turn, implies that $\drj = \dvj$ and
  the inequality \eqref{eq:drj_vs_dvj} is no longer needed. It
  simplifies the lower bound \eqref{eq:low_est} and yields the desired
  inequality.
\end{proof}

\section{Proof of Theorem \ref{thm:main}: Equality of Degrees}\label{sec:equality_of_degrees}

\begin{proof}
  To finish the proof of Theorem \ref{thm:main} we need to show that
  all $\dvj = 0$. We start with the quadratic inequality for $\dvj$
  obtained in Lemma \ref{lem:absence_of_bulk_vortices}:
  \begin{equation}\label{eq:decomp_ineq_4}
    \sum_{j=1}^N
    \Bra{ \pi(1 - \theta)^2(|\log{\d}| - |\log{R'}| + O(\d))\dvjp^2
      - 2\pi \dvj (\he - \Hrj) }
    \leq O(1),
  \end{equation}
  where $\Hrj = \he\xi_0(\aj) + \dj K_0(R)$. This inequality has the
  same structure as the quadratic functional in $S^1$-valued case:
  there are no mixed terms $\dvi\dvj$. Therefore we can find zeros for
  each $j = 1\dots N$ separately.
	
  Fix $1 \leq j \leq N$. Clearly, $\dvj = 0$ is one of two roots of
  \begin{equation}\label{eq:quad_for_fixed_j}
    \pi(1 - \theta)^2(|\log{\d}| - |\log{R'}| + O(\d))\dvjp^2
    - 2\pi \dvj (\he - \Hrj) = 0.
  \end{equation} Since $K_0(R) = |\log\d| + O(1)$ and
  \begin{equation}
    \dj = \left\llbracket\sigma\Par{1 - \xi_0(\aj)}\right\rrbracket,
  \end{equation}
  we can calculate the coefficient for the linear term in
  \eqref{eq:decomp_ineq_4}:
  \begin{align}
    -2\pi(\he - \Hrj)
    \nonumber &= -2\pi(\sigma|\log\d| - \sigma|\log\d|\xi_0(\aj)
                - \left\llbracket\sigma\Par{1 - \xi_0(\aj)}\right\rrbracket|\log\d|) + O(1) \\
              &= -2\pi|\log\d|\big(\sigma\Par{1 - \xi_0(\aj)} - \left\llbracket\sigma\Par{1 - \xi_0(\aj)}\right\rrbracket\big) + O(1).
  \end{align}
  Since $\left\llbracket\cdot\right\rrbracket$ is the nearest integer,
  we have
  \begin{equation}
    \Big|\sigma\Par{1 - \xi_0(\aj)} - \left\llbracket\sigma\Par{1 - \xi_0(\aj)}\right\rrbracket\Big|
    \leq \frac12 - \xi,
  \end{equation}
  assuming the uniqueness condition \eqref{eq:uniqueness_cond} and
  taking
  \begin{equation}
    \xi = \min_{j = 1\dots N} \dist\Par{\sigma\Par{1 - \xi_0(\aj)}, \Z + \frac12} > 0.
  \end{equation} The second zero of \eqref{eq:quad_for_fixed_j} can be estimated as follows:
  \begin{equation}
    |t_j|
    = \Abs{\frac{-2\pi\big(\sigma\Par{1 - \xi_0(\aj)} - \left\llbracket\sigma\Par{1 - \xi_0(\aj)}\right\rrbracket\big) + o(1)}{\pi(1 - \theta)^2 + o(1)}}
    < \frac{1 - 2\xi}{(1 - \theta)^2 + o(1)} + o(1).
  \end{equation}
  Having $\xi$ fixed and $\d < \d_0$ sufficiently small, we can always
  take $\theta > 0$ small enough to make sure that $|t_j| < 1 - \xi$.
		
  Since $\dvj$ can take only integer values, if at least one $\dvj$ is
  nonzero, the left hand side of \eqref{eq:decomp_ineq_4} becomes
  strictly positive of order $O(\log\d)$. This contradiction finishes
  the proof of main theorem yielding
  \begin{equation}
    \dvj = 0
    \text{ or }
    \dedj = \dj
  \end{equation}
  for all $j = 1\dots N$.
\end{proof}

\section{Acknowledgements}
The work of LB, OR and VR was partially supported by NSF grant DMS-1405769. The authors would like to thank Valery Vinokur for stimulating discussions of physics relevant to this work.

\bibliographystyle{plain} \bibliography{GL}

\begin{thebibliography}{10}

\bibitem{AftSanSer01}
A.~Aftalion, E.~Sandier, and S.~Serfaty.
\newblock Pinning phenomena in the {G}inzburg--{L}andau model of
  superconductivity.
\newblock {\em Journal de math{\'e}matiques pures et appliqu{\'e}es},
  80(3):339--372, 2001.

\bibitem{AlaBro05}
S.~Alama and L.~Bronsard.
\newblock Pinning effects and their breakdown for a {G}inzburg--{L}andau model
  with normal inclusions.
\newblock {\em Journal of Mathematical Physics}, 46(9):095102, 2005.

\bibitem{AlaBro06}
S.~Alama and L.~Bronsard.
\newblock Vortices and pinning effects for the {G}inzburg--{L}andau model in
  multiply connected domains.
\newblock {\em Communications on pure and applied mathematics}, 59(1):36--70,
  2006.

\bibitem{AndBauPhi03}
N.~Andre, P.~Bauman, and D.~Phillips.
\newblock Vortex pinning with bounded fields for the {G}inzburg--{L}andau
  equation.
\newblock 20(4):705--729, 2003.

\bibitem{AydKac09}
H.~Aydi and A.~Kachmar.
\newblock Magnetic vortices for a {G}inzburg--{L}andau type energy with
  discontinuous constraint. {II}.
\newblock {\em Commun. Pure Appl. Anal.}, 8:977--998, 2009.

\bibitem{BerRyb13}
L.~Berlyand and V.~Rybalko.
\newblock Homogenized description of multiple {G}inzburg--{L}andau vortices
  pinned by small holes.
\newblock {\em Networks \& Heterogeneous Media}, 8(1), 2013.

\bibitem{BetBreHel94}
F.~Bethuel, H.~Brezis, and F.~H{\'e}lein.
\newblock {\em {G}inzburg--{L}andau Vortices}.
\newblock Springer, 1994.

\bibitem{BetBreHel93}
Fabrice Bethuel, Ha{\"\i}m Brezis, and Fr{\'e}d{\'e}ric H{\'e}lein.
\newblock Asymptotics for the minimization of a ginzburg-landau functional.
\newblock {\em Calculus of Variations and Partial Differential Equations},
  1(2):123--148, 1993.

\bibitem{ChaRic97}
S.~J. Chapman and G.~Richardson.
\newblock Vortex pinning by inhomogeneities in type-{II} superconductors.
\newblock {\em Physica D: Nonlinear Phenomena}, 108(4):397--407, 1997.

\bibitem{DosMis11}
M.~Dos~Santos and O.~Misiats.
\newblock {G}inzburg--{L}andau model with small pinning domains.
\newblock {\em Netw. Heterog. Media}, 6(4), 2011.

\bibitem{GinLan65}
V.~L. Ginzburg and L.~D. Landau.
\newblock Collected papers of {L}. {D}. {L}andau.
\newblock {\em ed. D. Ter Haar}, 1965.

\bibitem{IarBerRybVin13}
O.~Iaroshenko, V.~Rybalko, V.~Vinokur, and L.~Berlyand.
\newblock Vortex phase separation in mesoscopic superconductors.
\newblock {\em Scientific Reports}, 3, 2013.

\bibitem{Jer99}
R.~Jerrard.
\newblock Lower bounds for generalized {G}inzburg--{L}andau functionals.
\newblock {\em SIAM Journal on Mathematical Analysis}, 30(4):721--746, 1999.

\bibitem{Kac10}
A.~Kachmar.
\newblock Magnetic vortices for a {G}inzburg--{L}andau type energy with
  discontinuous constraint.
\newblock {\em ESAIM Control Optim. Calc. Var.}, 16:545--580, 2010.

\bibitem{San98}
E.~Sandier.
\newblock Lower bounds for the energy of unit vector fields and applications.
\newblock {\em Journal of Functional Analysis}, 152(2):379--403, 1998.

\bibitem{SanSer00}
E.~Sandier and S.~Serfaty.
\newblock Global minimizers for the {G}inzburg--{L}andau functional below the
  first critical magnetic field.
\newblock 17(1):119--145, 2000.

\bibitem{SanSer00-2}
E.~Sandier and S.~Serfaty.
\newblock On the energy of type-{II} superconductors in the mixed phase.
\newblock {\em Reviews in Mathematical Physics}, 12(09):1219--1257, 2000.

\bibitem{SanSer03}
E.~Sandier and S.~Serfaty.
\newblock {G}inzburg--{L}andau minimizers near the first critical field have
  bounded vorticity.
\newblock {\em Calculus of Variations and Partial Differential Equations},
  17(1):17--28, 2003.

\bibitem{SanSer07}
E.~Sandier and S.~Serfaty.
\newblock {\em Vortices in the Magnetic {G}inzburg--{L}andau Model}.
\newblock Birkh{\"a}user Boston, 2007.

\end{thebibliography}

\appendix

\section{Appendix. Gradient estimate}
\begin{lemma}\label{lem:grad_est}
  Let $u$ solve the Poisson equation with Dirichlet boundary
  conditions in $\Od = \O \setminus \cup_{j=1}^N \odj$ with
  $\odj = B(\aj, \d)$, that is
  \begin{equation}
    \begin{cases}
      -\Lap u = f &\text{ in } \Od, \\
      u = g &\text{ on } \bd\O, \\
      u = g_j &\text{ on } \bd\odj,
    \end{cases}
  \end{equation}
  where $g$ and $g_j$ are smooth functions that are defined in the
  whole of $\Od$. Then
  \begin{equation}\label{eq:full_grad_est}
    \|\nabla u\|_{L^{\infty}(\Od)}
    \leq C\Par{\frac{1}{\d}\|u\|_{L^{\infty}(\Od)} + \|f\|_{L^\infty(\Od)} + \|\Lap g\|_{L^\infty(\O)}
      + \d\sum_{j=1}^{N}\|\Lap g_j\|_{L^\infty(\Od)}}.
  \end{equation}
\end{lemma}
\begin{proof}
  The proof is based on lemmas A.1 and A.2 from
  \cite{BetBreHel93}. Consider the three cases: the point $x_0\in\Od$
  is far from the boundaries of $\bd\Od$, it is close to $\bd\O$, and
  it is close to $\bd\odj$ for some $j = 1\dots N$. The first case
  when $x_0\in K \subset\subset \Od$ is resolved in Lemma A.1
  \cite{BetBreHel93} and the second case, when $x_0$ is close to
  $\bd\O$, can be deduced from Lemma A.2 using
  $\widetilde{u} = u - g$. The results of both lemmas can be merged
  together in the following estimate:
  \begin{equation}\label{eq:simple_grad_est}
    |\nabla u(x_0)|
    \leq C\Par{\|u\|_{L^{\infty}} + \|f\|_{L^\infty} + \|\Lap g\|_{L^\infty}}
    \quad\text{a.e.}
  \end{equation}
  when $\dist(x_0, \bd\odj) > m > 0$ with some fixed $m$ independent
  of $\d$.
		
  The third case is specific to our setting. Let $x_0$ be close to one
  of the holes: $\dist(x_0, \bd\odj) \leq m$ for some $j = 1\dots
  N$. Without loss of generality assume $\aj = 0$. We introduce the
  new spatial variable $y = \frac{x}{\d}$ to rescale the domain so
  that the $\odj$ becomes $B(0, 1)$ and $x_0$ becomes $y_0$. The
  Poisson equation in new coordinates becomes
  \begin{equation}
    -\Lap_y u = \d^2 f.
  \end{equation} If $\dist(y_0, \bd B(0, 1)) > m$, we apply Lemma A.1 from \cite{BetBreHel93} again. It gives us the estimate for $|\nabla_y u(y_0)|$:
  \begin{equation}
    |\nabla_y u(y_0)|
    \leq C\Par{\|u\|_{L^{\infty}} + \d^2\|f\|_{L^\infty}}
  \end{equation}
  that in turn implies the estimate for $|\nabla_x u(x_0)|$:
  \begin{equation}
    |\nabla_x u(x_0)|
    = \frac{1}{\d}|\nabla_y u(y_0)|
    \leq \frac{C}{\d}\|u\|_{L^{\infty}(\Od)} + C\d\|f\|_{L^\infty(\Od)}.
  \end{equation}
  Finally, we apply Lemma A.2 to $\widetilde{u}_j = u - g_j$ that
  satisfies the problem
  \begin{equation}
    \begin{cases}
      -\Lap_y \widetilde{u}_j = \d^2f + \Lap_y g_j &\text{ in } B(0, 1+m) \setminus \overline{B(0, 1)}, \\
      \widetilde{u}_j = h_j &\text{ on } \bd B(0, 2+m), \\
      \widetilde{u}_j = 0 &\text{ on } \bd B(0, 1).
    \end{cases}
  \end{equation}
  where $h_j(y) = u(y) - g_j(y)$. Since the proof of Lemma A.2 uses
  only local estimates and $y_0$ is far from the $\bd B(0, 2+m)$, the
  function $h_j$ does not play a role for the estimate of
  $|\nabla_y u(y_0)|$. It yields the estimate
  \begin{equation}
    |\nabla_y u(y_0)|
    \leq C\Par{\|u\|_{L^{\infty}} + \d^2\|f\|_{L^\infty} + \|\Lap_y g_j\|_{L^\infty}}.
  \end{equation}
  Going back to $x$ we obtain
  \begin{equation}
    |\nabla_x u(x_0)|
    \leq \frac{C}{\d}\|u\|_{L^{\infty}} + C\d\Par{\|f\|_{L^\infty} + \|\Lap_x g_j\|_{L^\infty}}.
  \end{equation}
  Merging all the estimates we finish the proof.
\end{proof}
\end{document}